\documentclass[11pt]{scrartcl}

\usepackage{mystyle-koma-local}

\newtheorem*{assumptionI}{Assumptions I}
\newtheorem*{assumptionII}{Assumptions II}
\newtheorem*{assumptionH}{Assumption H}
\newtheorem*{assumptionD}{Assumption D}

\newtheorem{question}[theorem]{Question}



\newcommand{\N}{{\mathbb{N}}}
\newcommand{\Z}{{\mathbb{Z}}}
\newcommand{\R}{{\mathbb{R}}}
\newcommand{\E}{{\mathbb{E}}}
\newcommand{\PP}{{\mathbb{P}}}
\newcommand{\1}{{\mathbf 1}}
\newcommand{\depth}{{\mathrm{depth}}}
\newcommand{\Cov}{{\mathrm{Cov}}}


\title{Partial orders and monotonicity of logarithmic depth and height in preferential attachment trees}
\author{Christian M\"onch}
\date{August 24, 2026}
\hypersetup{
  pdftitle={Partial orders and monotonicity of logarithmic depth and height in preferential attachment trees},
  pdfauthor={Christian M\"onch},
  pdfsubject={Preferential attachment trees},
  pdfkeywords={preferential attachment, random graphs, branching processes}
}

\begin{document}
\maketitle

\begin{abstract}
We study preferential attachment trees with general attachment rules. A natural intuition is that rewarding high-degree vertices more strongly should make the tree shallower. We test this intuition for the logarithmic growth constants of insertion depth and tree height.
We show by an explicit counterexample that growth-ratio dominance alone does not guarantee either monotonicity. We identify additional tail-order conditions, arising from the branching-process representation of the tree, under which both expected comparisons hold.
\end{abstract}

\medskip

\section{Introduction and main result}

\subsection*{General preferential attachment trees}
Fix an \emph{attachment function} $f:\N_0\to(0,\infty)$. We define a random rooted tree $T_n(f)$ on vertex set $[n]=\{1,2,\dots,n\}$ recursively as follows.

\begin{itemize}[leftmargin=2em]
  \item \textbf{Initialisation:} $T_1(f)$ consists of a single vertex $1$, the \emph{root}.
  \item \textbf{Recursion:} Given $T_{n-1}(f)$, obtain $T_n(f)$ by adding vertex $n$ and connecting it by a single edge to a vertex $i\in\{1,\dots,n-1\}$ chosen conditionally on $T_{n-1}(f)$ with probability
  \begin{equation*}
    \PP(n\to i\mid T_{n-1}(f))
    \;=\;
    \frac{f(\deg_{T_{n-1}(f)}(i))}
         {\sum_{j=1}^{n-1} f(\deg_{T_{n-1}(f)}(j))}.
  \end{equation*}
\end{itemize}

Here $\deg_{T_{n-1}(f)}(i)$ denotes the number of children of $i\in[n-1]$ in the rooted tree $T_{n-1}(f)$ (i.e.\ the out-degree in the orientation away from the root). In particular, every new vertex is born with $0$ children, so the value $f(0)$ governs the \emph{initial attractiveness} of a newborn vertex. (If $\deg_T^{\mathrm{gr}}(i)$ denotes the usual undirected graph degree, then $\deg_T^{\mathrm{gr}}(i)=\deg_T(i)+\1_{\{i\neq 1\}}$.)
This model interpolates between \emph{uniform attachment} ($f\equiv 1$) and \emph{affine preferential attachment} ($f(k)=k+\delta$), and was introduced in 2007 by Rudas, T\'oth, and Valk\'o~\cite{Rudas2007}.
Preferential attachment trees (and, more generally, preferential attachment networks) are fundamental objects in random graph theory and have been studied since early work of Szyma{\'n}ski~\cite{Szyman1985}, and especially following the Barab\'asi--Albert paradigm~\cite{BA1999,Albert2002}.

We study two logarithmic functionals: the \emph{insertion depth} of the newest vertex and the \emph{height} of the tree. For a rooted tree $T$ on $[n]$ let $\depth_T(v)$ denote the graph distance from the root to $v$, and set
\[
D_n(f) := \depth_{T_n(f)}(n),
\qquad
H_n(f) := \max_{v\in[n]} \depth_{T_n(f)}(v).
\]
Both insertion depth and height of random recursive trees are classical objects of study; see e.g.\ the monograph by Drmota~\cite{Drmota2009} for an overview. In particular, it is well known that for uniform attachment ($f\equiv 1$) both $D_n(f)$ and $H_n(f)$ are of order $\log n$ with explicit constants~\cite{Devroye1987,Devroye1988,Pittel1994}. For general PA trees, logarithmic order of $D_n(f)$ under mild regularity conditions follows from the Crump--Mode--Jagers branching process embedding~\cite{Rudas2007}, which can also be used to determine the logarithmic order of $H_n(f)$ in various regimes~\cite{Pittel1994,Drmota2009height,bhamidi2007universal}. For a recent treatment of depths in the broader setting of random recursive metric spaces, see Desmarais~\cite{Desmarais2024}.

In preferential attachment models which allow cycles, related functionals are the hop-count (corresponding to insertion depth) and the diameter (corresponding to height), which have been investigated for a variety of attachment regimes of both affine and general attachment functions~\cite{BolRio2004,Dommers2010,DMM12,DMM17,vanderhofstad2025}.

\subsection*{Order of attachment functions and monotonicity}
Throughout we consider \emph{nondecreasing} attachment functions $f,g:\N_0\to(0,\infty)$.
We say that $f$ \emph{dominates} $g$ in \emph{growth-ratio order}, and write
$f\succeq_{\mathrm{GR}} g$, if
\[
  \frac{f(k+1)}{f(k)} \;\ge\; \frac{g(k+1)}{g(k)}
  \qquad \text{for all } k\in\N_0.
\]
Equivalently, the ratio $h(k):=f(k)/g(k)$ is nondecreasing in $k$.

Intuitively, stronger reinforcement towards high-degree vertices should produce a shallower tree. Growth-ratio dominance (GRD) is a natural partial order for this intuition.
Set
\[
\E_f[D_n] := \E\big[D_n(f)\big],
\qquad
\E_f[H_n] := \E\big[H_n(f)\big].
\]

\begin{question}\label{ques:GRD-monotone}
Let $f,g:\N_0\to(0,\infty)$ be nondecreasing and assume $f\succeq_{\mathrm{GR}} g$.
Under which additional structural conditions, if any, does one have
\[
  \E_f[D_n] \;\le\; \E_g[D_n], \text{ and }  \E_f[H_n] \;\le\; \E_g[H_n].
\]
\end{question}

We prove conditional asymptotic forms of Question~\ref{ques:GRD-monotone} and show that the unconditional statement is false for both depth and height at the level of logarithmic constants.
Under the analytic and probabilistic conditions collected below the limits
\[
  \mathsf{c}_f := \lim_{n\to\infty}\frac{\E_f[D_n]}{\log n}, \quad \mathsf{c}^{\ast}_f := \lim_{n\to\infty}\frac{\E_f[H_n]}{\log n}
\]
exist and admit representations as parameters of the branching process embedding of $T_n(f)$; see Section~\ref{sec:cmj-parameters} for a brief exposition. For depth we state the required mean identification explicitly; for height we assume an $L^1$ law so that its expectation constant agrees with the probabilistic height constant.

We collect the additional hypotheses used in our main results in three groups:
\begin{itemize}[leftmargin=2em]
  \item Assumptions I, items (I.1)--(I.6), are generic CMJ/BRW regularity conditions on a single attachment function $f$.
  \item Assumptions II, items (II.1)--(II.3), are analytic conditions on the multiplicative interpolation $f_\theta$ between $g$ and $f$.
  \item Assumptions H and D are the profile-order conditions used in the height and depth comparison arguments, respectively.
\end{itemize}

\begin{assumptionI}
\label{ass:reg-attach}
For an attachment function $f:\N_0\to(0,\infty)$ let $m_f$ be the Laplace transform of the associated pure-birth process, cf.\ \eqref{eq:mf-def}--\eqref{eq:mf-series}. Whenever the Euler--Lotka equation $m_f(\lambda)=1$ has a unique solution we denote it by $\lambda_f$, and we set $Q_f:=-\lambda_f m_f'(\lambda_f)$ when finite.
\begin{enumerate}[label=(I.\arabic*)]
  \item (Monotonicity) $f$ is nondecreasing and strictly positive on $\N_0$.
  \item (Non-explosion) $\sum_{k\ge 1} 1/f(k)=\infty$. Equivalently, the pure-birth process with jump rates $f(k)$ does not explode.
  \item (Malthusian regime) There exists $\lambda>0$ with $m_f(\lambda)<\infty$, and then there is a unique $\lambda_f>0$ such that $m_f(\lambda_f)=1$.
  \item (Finite derivative) $0<-m_f'(\lambda_f)<\infty$.
  \item (Mean-depth identification) The expected insertion depth satisfies $\E_f[D_n(f)]/\log n\to Q_f^{-1}$. Thus the limit $\mathsf{c}_f$ exists and equals $1/Q_f$.
  \item (Height LLN/BRW speed) The CMJ/BRW embedding associated with $f$ satisfies the Kingman~\cite{Kingman1975} hypotheses (in particular finite exponential moments) ensuring a deterministic extremal speed $\kappa_f$, the variational representation \eqref{eq:kappa-var} with an interior maximiser $\lambda^{\ast}_f$, and $L^1$ convergence of the scaled height. Under these hypotheses
  \begin{equation*}
	\frac{H_n(f)}{\log n}\to \mathsf{c}^{\ast}_f=\frac{1}{\lambda_f\kappa_f}\text{ in }L^1\text{ and hence in probability.}
\end{equation*}
\end{enumerate}
\end{assumptionI}

We also need the following analytic regularity conditions along interpolations between attachment functions $f$ and $g$.
\begin{assumptionII}
\label{ass:reg-interp}
Let $f,g:\N_0\to(0,\infty)$ be nondecreasing with $f\succeq_{\mathrm{GR}}g$, set $h:=f/g$, and define the interpolation $f_\theta:=g h^\theta$ for $\theta\in[0,1]$.
\begin{enumerate}[label=(II.\arabic*)]
  \item (Uniform Malthusian gap) Along the interpolation the Euler--Lotka equation $m_{f_\theta}(\lambda_\theta)=1$ has a unique solution for each $\theta\in[0,1]$, and $\inf_{\theta\in[0,1]}\lambda_\theta\ge \lambda_->0$.
  \item (Uniform dominated envelope) For each $\theta$ and $\lambda>0$, set
\[
A_n^{(\theta)}(\lambda):=\prod_{i=0}^{n-1}\frac{f_\theta(i)}{f_\theta(i)+\lambda},
\qquad
r_i^{(\theta)}(\lambda):=\sum_{n\ge i+1}A_n^{(\theta)}(\lambda).
\]
There exists $\lambda_0\in(0,\lambda_-)$ such that
\begin{equation}\label{eq:uniform-dominating}
\sup_{\theta\in[0,1]}\sum_{i\ge 0}\frac{(1+|\log h(i)|)\,r_i^{(\theta)}(\lambda_0)}{f_\theta(i)+\lambda_0}<\infty.
\end{equation}
  \item (Height envelope regularity) Writing the same interpolation as $f_s:=g h^s$, define
  \[
  J_s(\lambda):=\frac{-\log m_{f_s}(\lambda)}{\lambda},
  \qquad \lambda>\lambda_s.
  \]
  Along the interpolation the variational problem
  \[
  \kappa_s:=\sup_{\lambda>\lambda_s}J_s(\lambda)
  \]
  has an interior maximiser $\lambda_s^\ast\in(\lambda_s,\infty)$, which may be chosen measurably in $s$.
  The map $s\mapsto\kappa_s$ is absolutely continuous and, for almost every $s\in[0,1]$, satisfies
  \[
  \kappa_s'
  =
  -\frac{1}{\lambda_s^\ast}\,\partial_s\log m_{f_s}(\lambda_s^\ast).
  \]
  The differentiations and sum interchanges at $\lambda_s^\ast$ are justified by the domination in (II.2).
\end{enumerate}
\end{assumptionII}

The assumptions above are regularity and analytic hypotheses: they ensure that the CMJ constants are well defined and that the interpolation can be differentiated. The next two assumptions are the profile-order inputs that drive the monotonicity comparisons themselves.

\begin{assumptionH}
\label{ass:height-endpoint-order}
Let $f,g:\N_0\to(0,\infty)$ be nondecreasing with $f\succeq_{\mathrm{GR}}g$, set $h:=f/g$, and write
$f_s:=g h^s$ for $s\in[0,1]$.
Let $\lambda_s$ be the Malthusian parameter of $f_s$, and let $\lambda_s^\ast$ be the height optimiser from (II.3).
For $s\in[0,1]$ and $\lambda\ge\lambda_s$, define
\[
A_n^{(s)}(\lambda):=\prod_{i=0}^{n-1}\frac{f_s(i)}{f_s(i)+\lambda},
\qquad
m_s(\lambda):=\sum_{n\ge1}A_n^{(s)}(\lambda),
\]
and let $N$ have distribution
\[
  \PP_{s,\lambda}(N=n):=\frac{A_n^{(s)}(\lambda)}{m_s(\lambda)},\qquad n\ge1.
\]
Set
\[
T_k^{(s)}(\lambda):=\PP_{s,\lambda}(N>k),\qquad
c_k^{(s)}(\lambda):=\frac{1}{f_s(k)+\lambda},\qquad
w_k^{(s)}(\lambda):=\lambda c_k^{(s)}(\lambda)T_k^{(s)}(\lambda),
\]
Finally, set
\[
W_s(\lambda):=\sum_{j\ge0}w_j^{(s)}(\lambda),
\qquad
\bar L_s(\lambda):=
\frac{1}{W_s(\lambda)}
\sum_{j\ge0}w_j^{(s)}(\lambda)\log h(j).
\]
The \emph{height-score endpoint order} is that, for every $s\in[0,1]$,
\[
  \bar L_s(\lambda_s)\ge \bar L_s(\lambda_s^\ast).
\]
Equivalently, the $\log h$-mean under the probability measure proportional to
$w_k^{(s)}(\lambda_s)$ is at least its $\log h$-mean under the probability measure proportional to
$w_k^{(s)}(\lambda_s^\ast)$.
\end{assumptionH}

\begin{assumptionD}
\label{ass:depth-profile-order}
Let $f,g:\N_0\to(0,\infty)$ be nondecreasing with $f\succeq_{\mathrm{GR}}g$, set $h:=f/g$, and write
$f_\theta:=g h^\theta$, $\theta\in[0,1]$. Let $\lambda_\theta$ be the Malthusian parameter from (II.1) and define the gauged interpolation
\[
  f_\theta^\ast(k):=\frac{f_\theta(k)}{\lambda_\theta}.
\]
For $\lambda=1$, set
\[
A_n^{(\theta,\ast)}:=\prod_{i=0}^{n-1}\frac{f_\theta^\ast(i)}{f_\theta^\ast(i)+1},
\qquad
r_k^{(\theta,\ast)}:=\sum_{n\ge k+1}A_n^{(\theta,\ast)},
\]
and
\[
\alpha_k^{(\theta)}:=\frac{1}{f_\theta^\ast(k)+1},
\qquad
w_k^{(\theta)}:=\alpha_k^{(\theta)}r_k^{(\theta,\ast)}.
\]
Define the corrected depth profile
\[
\Gamma_k^{(\theta)}
:=
\alpha_k^{(\theta)}
+
\frac{1}{r_k^{(\theta,\ast)}}
\sum_{j\ge0}\alpha_j^{(\theta)}r_{\max\{k,j\}}^{(\theta,\ast)}
\]
and its $w$-mean
\[
\overline\Gamma_\theta
:=
\frac{\sum_{j\ge0}w_j^{(\theta)}\Gamma_j^{(\theta)}}{\sum_{j\ge0}w_j^{(\theta)}}.
\]
The sums defining $\overline\Gamma_\theta$ are required to be finite.
The \emph{depth-profile order} is that, for every $\theta\in[0,1]$ and every $\ell\ge1$,
\begin{equation}\label{eq:depth-profile-order}
\sum_{k\ge\ell}
w_k^{(\theta)}
\bigl(\Gamma_k^{(\theta)}-\overline\Gamma_\theta\bigr)
\ge0.
\end{equation}
Equivalently, the probability measure proportional to $w_k^{(\theta)}\Gamma_k^{(\theta)}$ stochastically dominates the probability measure proportional to $w_k^{(\theta)}$.
\end{assumptionD}

Informally, GRD supplies an increasing direction: along the interpolation $f_s=g(f/g)^s$, the logarithmic derivative is $\partial_s\log f_s(k)=\log(f(k)/g(k))$, and this sequence is nondecreasing exactly when $f\succeq_{\mathrm{GR}}g$.  Assumptions~D and H are the extra statements that the depth and height sensitivities have the right sign when tested against this increasing direction.  Section~\ref{sec:order-geometry} develops this geometric viewpoint after the CMJ notation has been introduced.

Our first theorem shows that $\mathsf{c}_f$ is monotone under GRD once the corrected depth profile has the required stochastic order.

\begin{theorem}\label{thm:monotonicity}
Let $f,g:\N_0\to(0,\infty)$ be nondecreasing with $f \succeq_{\mathrm{GR}} g$.
Assume that $f$ and $g$ satisfy Assumptions~I (I.1)--(I.5), that the multiplicative interpolation between $g$ and $f$ satisfies Assumptions~II (II.1)--(II.2), and that Assumption~D holds.
Then the limits $\mathsf{c}_f,\mathsf{c}_g$ exist and
\[
  \mathsf{c}_f \le \mathsf{c}_g.
\]
\end{theorem}

The height result is analogous, with the height-score endpoint order replacing Assumption~D.
\begin{theorem}\label{thm:height-monotonicity}
Let $f,g:\N_0\to(0,\infty)$ be nondecreasing with $f\succeq_{\mathrm{GR}} g$, and write $f_s:=g\,(f/g)^s$ for the GRD interpolation, $s\in[0,1]$.
Assume that along $(f_s)_{s\in[0,1]}$ the functions satisfy Assumptions~I (I.1)--(I.4) and (I.6).
Assume moreover that the interpolation satisfies Assumptions~II (II.1)--(II.3), and that the pair $(g,f)$ satisfies
Assumption~H.
Then the limits $\mathsf{c}^{\ast}_f,\mathsf{c}^{\ast}_g$ exist and
\[
  \mathsf{c}^{\ast}_f\le \mathsf{c}^{\ast}_g.
\]
\end{theorem}

\subsection*{Guide to the paper}
Section~\ref{sec:cmj-parameters} introduces the standard Crump--Mode--Jagers (CMJ) embedding of $T_n(f)$ and identifies the logarithmic depth and height constants with CMJ/BRW parameters.  Section~\ref{sec:order-geometry} then develops the order-theoretic viewpoint behind Assumptions~D and H: the dual-cone lemma, the common tail-order formulation, the finite-dimensional counterexample to unconditional GRD monotonicity, and concrete verification examples.

The proofs of the two monotonicity theorems are separated by observable.  Section~\ref{sec:depth-proof} proves Theorem~\ref{thm:monotonicity} by interpolating between $g$ and $f$, applying a gauge normalisation, and reducing the sign of $Q_\theta'$ to the depth-profile tail inequalities through Abel summation.  Section~\ref{sec:height-proof} proves Theorem~\ref{thm:height-monotonicity} using the BRW speed functional $\kappa_f$ and its optimiser $\lambda^{\ast}_f$. Appendix~\ref{app:analytic-regularity} contains the analytic identities and domination arguments used to justify the differentiations and limiting operations.

\section{CMJ embedding and parameter identification}
\label{sec:cmj-parameters}

This section records the standard continuous-time embedding of $T_n(f)$ into a supercritical
Crump--Mode--Jagers (CMJ) branching process, and the corresponding identification of the
logarithmic depth and height constants with CMJ/BRW parameters. All results stated here are
classical; for preferential-attachment trees in this generality see e.g.\ Rudas et al.~\cite{Rudas2007} and the manuscript by Bhamidi~\cite{bhamidi2007universal},
and for CMJ background see the classical works by Jagers and Nerman~\cite{Jagers1975,Nerman1981,JagersNerman1984} and the more recent survey by Holmgren and Janson~\cite{HolmgrenJanson2017}.

\subsection*{Continuous-time embedding}
Fix an attachment function $f:\N_0\to(0,\infty)$. Consider the continuous-time rooted tree process
$(\Upsilon_f(t))_{t\ge 0}$ started from a single root at time $0$, in which (conditionally on
$\Upsilon_f(t)$) each vertex $x$ gives birth to a new child at instantaneous rate
$f(\deg_{\Upsilon_f(t)}(x))$, independently across vertices. Let
\[
\sigma_n := \inf\{t\ge 0:\ |\Upsilon_f(t)| = n\},\qquad n\ge 1,
\]
be the $n$th birth time. This is the standard ``competing exponentials'' argument for continuous-time branching constructions~\cite{Rudas2007,HolmgrenJanson2017}.

\begin{lemma}\label{lem:ct-embedding}
The discrete-time process $(\Upsilon_f(\sigma_n))_{n\ge 1}$ has the same law as the preferential attachment
tree process $(T_n(f))_{n\ge 1}$.
\end{lemma}
For self-containedness we include the short proof in Appendix~\ref{app:continuous-time-embedding}.

\subsection*{Laplace transform and CMJ parameters}

Fix $f:\N_0\to(0,\infty)$ and let $(X_f(t))_{t\ge0}$ be the pure-birth process on $\Z_{\ge0}$ started from
$X_f(0)=0$, with jump rates $k\to k+1$ at rate $f(k)$. Write $\tau_n$ for the time of the $n$th jump and
define the reproduction point process
\[
\xi_f \;:=\; \sum_{n\ge 1}\delta_{\tau_n}.
\]
In the CMJ tree $(\Upsilon_f(t))_{t\ge0}$ each vertex reproduces according to an independent copy of $\xi_f$,
with birth times measured relative to its own birth time. Thus $(\Upsilon_f(t))_{t\ge0}$ is a one-type CMJ
process with reproduction process $\xi_f$.

For $\lambda>0$ define the Laplace transform
\begin{equation}\label{eq:mf-def}
m_f(\lambda)
:=
\E\Big[\int_{[0,\infty)} e^{-\lambda t}\,\xi_f(dt)\Big]
=
\E\Big[\sum_{n\ge 1} e^{-\lambda \tau_n}\Big].
\end{equation}
Since $\tau_n=\sum_{i=0}^{n-1}E_i$ with independent $E_i\sim \mathrm{Exp}(f(i))$, one obtains the series form
\begin{equation}\label{eq:mf-series}
m_f(\lambda)
=
\sum_{n\ge 1}\ \prod_{i=0}^{n-1}\frac{f(i)}{f(i)+\lambda}.
\end{equation}
Where finite, $\lambda\mapsto m_f(\lambda)$ is strictly decreasing. In the supercritical Malthusian regime
(standard CMJ hypotheses) there is a unique $\lambda_f>0$ solving the Euler--Lotka equation
\[
m_f(\lambda_f)=1,
\]
called the \emph{Malthusian parameter}~\cite{Jagers1975}.

Assuming differentiability at $\lambda_f$ and finiteness of the derivative, set
\begin{equation*}
Q_f := -\lambda_f m_f'(\lambda_f).
\end{equation*}

\begin{remark}
In CMJ language, $m_f(\lambda)=\int_0^\infty e^{-\lambda t}\,\nu_f(dt)$ is the Laplace transform of the mean
reproduction measure $\nu_f(dt):=\E[\xi_f(dt)]$, and $-m_f'(\lambda_f)=\int_0^\infty t\,e^{-\lambda_f t}\,\nu_f(dt)$ is the
``mean age at childbearing'' under the Malthusian tilt.
\end{remark}

Convergence of the Malthusian martingale controls population growth in time, but does not by itself identify the mean generation of the $n$th birth. We therefore state the required mean-depth identification separately in (I.5). The analytic ingredient used in the comparison proof is the following tail identity.
\begin{lemma}\label{lem:cmj-depth-parameters}
Suppose that $m_f(\lambda_f)=1$ has a unique solution and that $0<-m_f'(\lambda_f)<\infty$. Then $Q_f:=-\lambda_fm_f'(\lambda_f)\in(0,\infty)$ and
\begin{equation}\label{eq:mfprime-tail}
-m_f'(\lambda_f)
=\sum_{i\ge 0}\frac{r_i^{(f)}}{f(i)+\lambda_f},
\qquad
r_i^{(f)}
:=\sum_{n\ge i+1}\,\prod_{k=0}^{n-1}\frac{f(k)}{f(k)+\lambda_f}.
\end{equation}
\end{lemma}

\begin{remark}\label{rem:normalise-lambda}
The discrete PA law (and thus $Q_f$) is invariant under scaling: if $(cf)(k):=c\,f(k)$ with $c>0$, then
$\lambda_{cf}=c\,\lambda_f$ and $Q_{cf}=Q_f$. In particular, one may normalise to $\lambda_f=1$ without changing $Q_f$.
\end{remark}

We next recall the standard reduction of the CMJ height to an extremal problem in a branching random walk
(BRW). This viewpoint is classical for age-dependent branching processes~\cite{Kingman1975,Biggins1976,Biggins1977,HolmgrenJanson2017}. For more applications to increasing/PA-type trees, see e.g.\
Broutin and Devroye~\cite{BroutinDevroye2006}, Broutin et al.~\cite{BroutinDevroyeMcLeishdeLaSalle2008} and Janson~\cite{Janson2019}.

Let $Z_f(t):=|\Upsilon_f(t)|$ be the CMJ population size. Under the standard Jagers--Nerman conditions for the population characteristic, including the relevant $x\log x$ integrability, one has the almost sure non-degenerate limit~\cite{Nerman1981,JagersNerman1984}
\begin{equation}\label{eq:cmj-growth-height}
e^{-\lambda_f t}Z_f(t)\longrightarrow \mathcal W_f,
\qquad t\to\infty.
\end{equation}
Recall the birth time of the $n$th individual
\begin{equation*}
\sigma_n \;=\; \inf\{t\ge 0:\ Z_f(t)\ge n\}.
\end{equation*}
Then \eqref{eq:cmj-growth-height} yields the logarithmic size--time correspondence
\begin{equation}\label{eq:sigma-asymp-height}
\sigma_n \;=\; \frac{1}{\lambda_f}\log n \;+\; o(\log n),
\qquad n\to\infty.
\end{equation}

For an individual $v$ in the CMJ genealogy, write $|v|$ for its generation and $\mathrm{birth}(v)$ for its birth time.
Define the CMJ height at time $t$ by
\begin{equation*}
H_f(t)\;:=\;\max\{|v|:\ \mathrm{birth}(v)\le t\}.
\end{equation*}
Since the discrete PA tree on $n$ vertices is the CMJ genealogy observed at time $\sigma_n$, we have the identity
\begin{equation}\label{eq:Hn-equals-Ht-height}
H_n(f)\;=\; H_f(\sigma_n),
\qquad n\ge1.
\end{equation}

Let now $\mathcal B_k := \{\mathrm{birth}(v):\ |v|=k\}$ be the multiset of birth times in generation $k$.
Then $(\mathcal B_k)_{k\ge0}$ forms a BRW on $\R_+$ with displacement point process $\xi_f$.
Let
\begin{equation*}
\mathcal M_k
:=\inf\{t\ge0: \mathcal B_k\cap[0,t]\neq\emptyset\}
\end{equation*}
be the earliest birth time in generation $k$. Then
\begin{equation*}
H_f(t)\;=\;\max\{k\ge 0:\ \mathcal M_k\le t\}.
\end{equation*}

Under standard BRW regularity, the leftmost position has deterministic speed: there exists $\kappa_f\in(0,\infty)$ such that
\begin{equation}\label{eq:Mk-speed-height}
\frac{\mathcal M_k}{k}\longrightarrow \kappa_f,
\qquad k\to\infty,
\end{equation}
and $\kappa_f$ admits a variational representation in terms of $m_f$, for the details see e.g.\ Biggins~\cite{Biggins1976} or McDiarmid~\cite{McDiarmid1995}.

\begin{lemma}\label{lem:height-variational}
Assume the BRW regularity hypotheses ensuring \eqref{eq:Mk-speed-height}. For $\lambda>\lambda_f$ define
\begin{equation*}
J_f(\lambda)\;:=\;\frac{-\log m_f(\lambda)}{\lambda}.
\end{equation*} Then the leftmost-displacement speed $\kappa_f$
admits the variational representation
\begin{equation}\label{eq:kappa-var}
\kappa_f \;=\;\sup_{\lambda>\lambda_f} J_f(\lambda)
\;=\;\sup_{\lambda>\lambda_f}\frac{-\log m_f(\lambda)}{\lambda}.
\end{equation}
Moreover, in the regimes used below, the supremum in \eqref{eq:kappa-var} is assumed to be attained at an interior maximiser $\lambda^{\ast}_f\in(\lambda_f,\infty)$.
\end{lemma}

Combining \eqref{eq:Hn-equals-Ht-height}, \eqref{eq:sigma-asymp-height}, and the speed law \eqref{eq:Mk-speed-height} yields the height constant
\begin{equation*}
\frac{H_n(f)}{\log n}
\;\longrightarrow\;
\frac{1}{\lambda_f\kappa_f},
\qquad n\to\infty,
\end{equation*}
in probability, and in $L^1$ under Assumption~I (I.6).
We therefore set
\begin{equation*}
R_f \;:=\; \lambda_f\kappa_f,
\qquad\text{so that}\qquad
\mathsf{c}^{\ast}_f\;:=\;\frac{1}{R_f}.
\end{equation*}
Under the $L^1$ height hypothesis this is also
\[
\mathsf{c}^{\ast}_f=\lim_{n\to\infty}\frac{\E_f[H_n]}{\log n}.
\]
\begin{lemma}\label{lem:lamstar-characterisation}
Assume that $\lambda^{\ast}_f\in(\lambda_f,\infty)$ is an interior maximiser in \eqref{eq:kappa-var} and that $m_f$ is differentiable at $\lambda^{\ast}_f$.
Then $\lambda^{\ast}_f$ satisfies the stationary condition
\begin{equation*}
\log m_f(\lambda^{\ast}_f)\;=\;\lambda^{\ast}_f\,\frac{m_f'(\lambda^{\ast}_f)}{m_f(\lambda^{\ast}_f)},
\end{equation*}
and consequently
\begin{equation*}
\kappa_f
\;=\;
J_f(\lambda^{\ast}_f)
\;=\;
-\frac{m_f'(\lambda^{\ast}_f)}{m_f(\lambda^{\ast}_f)}
\;=\;
-\partial_\lambda\log m_f(\lambda^{\ast}_f).
\end{equation*}
\end{lemma}

\begin{remark}
In typical applications one first rescales $f$ and $g$ so that $\lambda_f=\lambda_g=1$ (Remark~\ref{rem:normalise-lambda}).
Under the regularity assumptions ensuring that $(\theta,\lambda)\mapsto m_{f_\theta}(\lambda)$ is finite and continuous on $[0,1]\times[\lambda_0,\infty)$ for some $\lambda_0>0$
and strictly decreasing in $\lambda$, the implicit equation $m_{f_\theta}(\lambda_\theta)=1$ defines a continuous map $\theta\mapsto \lambda_\theta$.
Since $\lambda_\theta>0$ for each $\theta$ and $[0,1]$ is compact, one may then take
\(
\lambda_-:=\min_{\theta\in[0,1]}\lambda_\theta>0
\)
in Assumptions~II (II.1) and (II.2).
\end{remark}

\section{Order geometry, counterexamples, and verification}
\label{sec:order-geometry}

This section explains the order-theoretic content of Assumptions~D and H using the CMJ notation from Section~\ref{sec:cmj-parameters}.  The guiding idea is simple: GRD gives an increasing perturbation direction, while the branching-process constants respond to that direction through signed first-variation measures.  The monotonicity proofs work when those signed measures have nonnegative pairing with increasing test functions.

\begin{lemma}\label{lem:increasing-dual-cone}
Let $\mu=(\mu_k)_{k\ge0}$ be a finite signed measure on $\N_0$ with total mass zero.  For a sequence $b=(b_k)_{k\ge0}$ set
\begin{equation}\label{eq:dual-cone-pairing}
\langle b,\mu\rangle:=\sum_{k\ge0}b_k\mu_k
\end{equation}
whenever this series is absolutely convergent, and write
\begin{equation}\label{eq:dual-cone-tail}
\mu([\ell,\infty)):=\sum_{k\ge\ell}\mu_k,\qquad \ell\ge1.
\end{equation}
Then the following are equivalent:
\begin{enumerate}[label=(\roman*)]
\item $\langle b,\mu\rangle\ge0$ for every nondecreasing sequence $b=(b_k)_{k\ge0}$ for which \eqref{eq:dual-cone-pairing} is absolutely convergent;
\item the upper tails of $\mu$ are nonnegative:
\[
\mu([\ell,\infty))\ge0,
\qquad \ell\ge1.
\]
\end{enumerate}
In particular, for probability measures $\rho$ and $\eta$ on $\N_0$,
\[
\rho\succeq_{\mathrm{st}}\eta
\quad\Longleftrightarrow\quad
\sum_{k\ge0}b_k(\rho_k-\eta_k)\ge0
\quad\text{for every bounded nondecreasing }b,
\]
and the same conclusion holds for every nondecreasing $b$ for which the displayed sum is absolutely convergent.
\end{lemma}

\begin{proof}
For a bounded nondecreasing $b$, write
$b_k=b_0+\sum_{\ell=1}^k(b_\ell-b_{\ell-1})$.  Since $\sum_k\mu_k=0$,
\[
\sum_{k\ge0}b_k\mu_k
=
\sum_{\ell\ge1}(b_\ell-b_{\ell-1})\,\mu([\ell,\infty)).
\]
This identity first holds for finitely supported truncations and then follows by dominated convergence, since $\mu$ is finite and $b$ is bounded.  Nonnegative tails therefore imply nonnegative pairing with every bounded nondecreasing $b$.

If $b$ is nondecreasing and \eqref{eq:dual-cone-pairing} is absolutely convergent, subtract the constant $b_0$ and truncate by $b^{(N)}_k:=b_{\min\{k,N\}}-b_0$.  The bounded case gives $\langle b^{(N)},\mu\rangle\ge0$, and absolute convergence gives $\langle b^{(N)},\mu\rangle\to\langle b,\mu\rangle$.  Conversely, taking $b_k=\mathbf 1_{\{k\ge\ell\}}$ gives the tail inequalities.  The probability-measure statement is the same assertion applied to $\mu=\rho-\eta$.
\end{proof}

\begin{remark}\label{rem:grd-tangent-duality}
The lemma is the order-theoretic content behind both monotonicity proofs.  Along the multiplicative interpolation
\[
f_s(k)=g(k)\bigl(f(k)/g(k)\bigr)^s,
\]
the tangent in logarithmic coordinates is
\[
b_k=\partial_s\log f_s(k)=\log\frac{f(k)}{g(k)}.
\]
Growth-ratio dominance says exactly that this tangent lies in the cone of nondecreasing sequences.

The role of Assumptions~D and H is to control the first variation of the CMJ/BRW constants against this increasing direction.  For depth the relevant signed measure is
\[
\mu^{D}_k
:=
w_k^{(\theta)}\bigl(\Gamma_k^{(\theta)}-\overline\Gamma_\theta\bigr),
\]
and Assumption~D is exactly $\mu^{D}([\ell,\infty))\ge0$ for every $\ell$, i.e. membership in the dual cone of increasing tests.  For height, Assumption~H is the pair-specific endpoint comparison: it tests the endpoint sensitivity only against the actual GRD tangent $\log h$.  A stronger, score-independent version is the stochastic endpoint order
\[
\nu_{s,\lambda_s}\succeq_{\mathrm{st}}\nu_{s,\lambda_s^\ast},
\qquad
\nu_{s,\lambda}(k):=\frac{w_k^{(s)}(\lambda)}{W_s(\lambda)},
\]
which implies Assumption~H because $k\mapsto\log h(k)$ is nondecreasing under GRD.

In this sense D and H are not stronger forms of GRD.  GRD selects an increasing perturbation direction, while D and H assert positivity of the branching sensitivity against increasing directions.  If one wants a local monotonicity statement uniformly for all increasing GRD tangents at a fixed attachment function, the corresponding dual-cone condition is also necessary: the step functions $\mathbf 1_{\{k\ge\ell\}}$ test exactly the upper-tail inequalities.
\end{remark}

\begin{remark}\label{rem:DH-tail-order}
There is a slightly broader way to view the two profile-order assumptions.  Fix a positive attachment function $\varphi$ and, for $\lambda$ in the domain of $m_\varphi$, write
\[
A_n(\lambda):=\prod_{i=0}^{n-1}\frac{\varphi(i)}{\varphi(i)+\lambda},
\qquad
r_k(\lambda):=\sum_{n\ge k+1}A_n(\lambda),
\qquad
c_k(\lambda):=\frac{1}{\varphi(k)+\lambda}.
\]
Set
\[
\nu_\lambda(k):=
\frac{\lambda c_k(\lambda)r_k(\lambda)}
{\sum_{j\ge0}\lambda c_j(\lambda)r_j(\lambda)}.
\]
The corrected $\lambda$--score of these weights is
\[
\mathcal G_k(\lambda)
:=
\lambda c_k(\lambda)
+
\frac{\lambda}{r_k(\lambda)}
\sum_{j\ge0}c_j(\lambda)r_{\max\{k,j\}}(\lambda),
\]
because
\[
\partial_\lambda\log\!\bigl(\lambda c_k(\lambda)r_k(\lambda)\bigr)
=
\frac{1}{\lambda}-\frac{\mathcal G_k(\lambda)}{\lambda}.
\]
Consequently, for every upper tail,
\[
\partial_\lambda\nu_\lambda([\ell,\infty))
=
-\frac{1}{\lambda}
\sum_{k\ge\ell}\nu_\lambda(k)
\bigl(\mathcal G_k(\lambda)-\E_{\nu_\lambda}\mathcal G(\lambda)\bigr).
\]
At the Malthusian point, with $\varphi$ equal to the interpolated attachment function, the profile $\mathcal G_k(\lambda_\varphi)$ is exactly the depth profile $\Gamma_k$ after the usual gauge normalisation.  Thus Assumption~D is the infinitesimal statement that the tails of $\nu_\lambda$ are decreasing as $\lambda$ passes through the Malthusian point.  Assumption~H is the endpoint version tested only against the increasing score $\log h$: it follows, for example, from the stochastic endpoint order $\nu_{\lambda_s}\succeq_{\mathrm{st}}\nu_{\lambda_s^\ast}$.  The stronger condition that $\lambda\mapsto\nu_\lambda$ is stochastically nonincreasing on $[\lambda_s,\lambda_s^\ast]$ would therefore imply both D and H, but it is not automatic from GRD alone.
\end{remark}

\begin{remark}\label{rem:common-profile-monotonicity}
Assumptions~H and D are the direct forms used in the two derivative arguments, but they have a common stronger verification route. In each case the proof produces a corrected profile, and monotonicity of that profile is enough.

For depth, if for every $\theta\in[0,1]$ the sequence
$k\mapsto\Gamma_k^{(\theta)}$ is nondecreasing, then Assumption~D follows: biasing the weights $w_k^{(\theta)}$ by the nondecreasing factor $\Gamma_k^{(\theta)}$ shifts the resulting probability measure to the right in stochastic order.

For height, if for every $s\in[0,1]$ and every $\lambda\in[\lambda_s,\lambda_s^\ast]$ the corrected height profile
$k\mapsto\widetilde M_k^{(s)}(\lambda)$ from Section~\ref{sec:height-proof} is nondecreasing, then the identity
$\partial_\lambda w_k^{(s)}(\lambda)=-w_k^{(s)}(\lambda)\widetilde M_k^{(s)}(\lambda)$ implies that
$\lambda\mapsto\nu_{s,\lambda}$ is stochastically nonincreasing. Since $k\mapsto\log h(k)$ is nondecreasing under GRD, the endpoint order in Assumption~H follows.

Thus one can verify both profile-order assumptions by the same principle: the relevant corrected regression profile should be increasing. The concrete one-step checks \eqref{eq:depth-onestep-sufficient} and \eqref{eq:onestep-ineq} below are two implementations of this principle.
\end{remark}

\begin{remark}\label{rem:grd-counterexample}
Let
\[
f_t(0)=1,\qquad f_t(1)=2,\qquad f_t(2)=1000t,\qquad f_t(k)=10^6t\quad(k\ge3).
\]
For $t\ge1$ the ratio $f_t/f_1$ is nondecreasing, hence $f_t\succeq_{\mathrm{GR}}f_1$, and these eventually constant attachment functions satisfy the analytic assumptions used above. A rational interval computation gives, for $t=101/100$,
\[
Q_{f_1}\ge 3.065668995785014984785127,
\qquad
Q_{f_{101/100}}\le 3.065210193098909320427347.
\]
The computation is finite-dimensional: with
\[
p_0=\frac{1}{1+\lambda},\qquad
p_1=\frac{2}{2+\lambda},\qquad
p_2=\frac{1000t}{1000t+\lambda},
\]
one has
\[
m_t(\lambda)=p_0+p_0p_1+p_0p_1p_2\,\frac{10^6t+\lambda}{\lambda},
\]
Exact rational bisection isolates the Euler--Lotka roots, and rational interval evaluation of $Q_t=-\lambda_t m_t'(\lambda_t)$ gives the display. Hence $Q$ decreases and $\mathsf{c}_f=1/Q_f$ increases along this GRD perturbation, disproving unconditional depth monotonicity.

The same example also obstructs height monotonicity without Assumption~H. With
\[
R_{f_t}:=\lambda_t\kappa_t,\qquad
\kappa_t:=\sup_{x>\lambda_t}\frac{-\log m_t(x)}{x},
\]
Strict log-convexity of $m_t$ makes
$F_t(x):=\log m_t(x)-x m_t'(x)/m_t(x)$ strictly decreasing, since $F_t'(x)=-x(\log m_t)''(x)<0$. Its unique zero is the optimiser, where $\kappa_t=-m_t'/m_t$; exact rational logarithm bounds isolate it and give
\[
R_{f_1}\ge 1.13769,
\qquad
R_{f_{101/100}}\le 1.13750.
\]
Thus $R$ decreases and $\mathsf{c}^{\ast}_f=1/R_f$ increases along the perturbation, so GRD alone does not imply height monotonicity. The accompanying file \texttt{counterexample\_certificate.py} certifies both displays without floating-point sign decisions.

The example also shows that Assumptions~D and H exclude real obstructions. Along the interpolation from $f_1$ to $f_{101/100}$, the ratio $h=f_{101/100}/f_1$ has a single jump at $k=2$, so the Abel identity in the depth proof gives
\[
Q_\theta'=\log(101/100)\,S_2^{(\theta)}.
\]
Assumption~D would make $Q_\theta$ nondecreasing, contradicting the depth bounds; Assumption~H and Theorem~\ref{thm:height-monotonicity} would make $R_{f_t}$ nondecreasing, contradicting the height bounds. Thus each order fails somewhere along this GRD interpolation.
\end{remark}

\begin{remark}\label{rem:affine-constants}
The `base case' of affine preferential attachment functions
\[
f_\delta(k)=k+\delta,\qquad k\in\mathbb N_0,\qquad \delta>0.
\]
satisfies the standard CMJ/BRW regularity assumptions and is a useful calibration case. In the equivalent language of generalized plane-oriented or linear recursive trees, the logarithmic depth/profile constant goes back to Bergeron--Flajolet--Salvy~\cite{BergeronFlajoletSalvy1992}, while the corresponding height constant for the scale-free family was obtained by Pittel~\cite{Pittel1994}. We include the short CMJ calculation below to fix our out-degree normalisation.

The sublinear verification criteria for Assumption~H below do not apply to affine attachment. Moreover, the multiplicative interpolation between two affine functions is no longer affine, so checking Assumptions~D and H along that interpolation is a separate matter. In the language of Remark~\ref{rem:DH-tail-order}, this would amount to proving the relevant tail monotonicity of the measures $\nu_\lambda$ along the non-affine bridge. We do not use such a statement here, and instead record the affine constants directly.

For every $\lambda>1$, set
\begin{align}
A^{(\delta)}_n(\lambda)
&:=\prod_{j=0}^{n-1}\frac{f_\delta(j)}{f_\delta(j)+\lambda}
=\prod_{j=0}^{n-1}\frac{j+\delta}{j+\delta+\lambda}
=\frac{\Gamma(n+\delta)\,\Gamma(\delta+\lambda)}{\Gamma(\delta)\,\Gamma(n+\delta+\lambda)},
\\[0.3em]
m_\delta(\lambda)
&:=\sum_{n\ge1}A^{(\delta)}_n(\lambda)
=\frac{\delta}{\lambda-1}.
\end{align}
The corresponding Malthusian parameter is
\begin{equation*}
\lambda_\delta=\delta+1.
\end{equation*}
Since $m_\delta'(\lambda)=-\delta/(\lambda-1)^2$, the depth parameter is
\[
Q_{f_\delta}
=-\lambda_\delta m_\delta'(\lambda_\delta)
=\frac{\delta+1}{\delta},
\qquad
\mathsf{c}_{f_\delta}
=\frac{\delta}{\delta+1}.
\]
Moreover, the BRW/height variational problem
\[
\kappa_\delta=\sup_{\lambda>\lambda_\delta}\frac{-\log m_\delta(\lambda)}{\lambda}
=\sup_{\lambda>\delta+1}\frac{\log\big(\frac{\lambda-1}{\delta}\big)}{\lambda}
\]
has a unique interior maximiser $\lambda^{\ast}_\delta\in(\delta+1,\infty)$ given by
\begin{equation*}
\lambda^{\ast}_\delta = 1+\frac{1}{W\big(\frac{1}{\delta e}\big)},
\qquad
\kappa_\delta = W\big(\tfrac{1}{\delta e}\big),
\end{equation*}
where $W$ denotes the principal branch of the Lambert $W$--function.
In particular,
\begin{equation*}
\mathsf{c}^{\ast}_{f_\delta}=\frac{1}{\lambda_\delta\,\kappa_\delta}
=\frac{1}{(\delta+1)\,W\big(\frac{1}{\delta e}\big)}.
\end{equation*}
The maps $\delta\mapsto\mathsf{c}_{f_\delta}$ and $\delta\mapsto \mathsf{c}^{\ast}_{f_\delta}$ are strictly increasing on $(0,\infty)$; for the height constant this follows, for instance, by differentiating $(\delta+1)W(1/(\delta e))$.
Thus, if $0<\delta_1<\delta_2$, then $f_{\delta_1}\succeq_{\mathrm{GR}}f_{\delta_2}$ and
$\mathsf{c}_{f_{\delta_1}}<\mathsf{c}_{f_{\delta_2}}$ and
$\mathsf{c}^{\ast}_{f_{\delta_1}}<\mathsf{c}^{\ast}_{f_{\delta_2}}$, so the affine depth and height constants obey the expected GRD monotonicity directly.
\end{remark}

\begin{remark}\label{rem:one-jump-example}
The assumptions are non-vacuous beyond the explicit affine endpoint formulas.  Consider the bounded attachment functions
\[
u_\tau(0)=1,\qquad
u_\tau(k)=\tau\quad(k\ge1),
\qquad \tau\ge1.
\]
If $1\le \tau_0<\tau_1$, then $u_{\tau_1}\succeq_{\mathrm{GR}}u_{\tau_0}$, and the multiplicative interpolation remains in the same family:
\[
u_{\tau_s}(0)=1,\qquad
u_{\tau_s}(k)=\tau_s\quad(k\ge1),
\qquad
\tau_s:=\tau_0(\tau_1/\tau_0)^s.
\]
This class is simple enough that the profile-order assumptions can be verified by hand.

For fixed $\tau$ put $x:=\sqrt{\tau}$, $q:=x/(1+x)$ and $a:=1/(1+x)$.  The Laplace transform is
\[
m_\tau(\lambda)
=
\sum_{n\ge1}A_n^{(\tau)}(\lambda)
=
\frac{\tau+\lambda}{\lambda(1+\lambda)},
\]
so the Malthusian parameter is $\lambda_\tau=x$.  In the ungauged notation at $\lambda=x$,
\[
r_k=q^k,\qquad
\alpha_0=q,\qquad
\alpha_k=a\quad(k\ge1).
\]
Consequently
\[
w_0=q,\qquad w_k=aq^k\quad(k\ge1),
\]
and the corrected depth profile is
\[
\Gamma_0=3q,\qquad
\Gamma_k=2+(k-1)a\quad(k\ge1),
\qquad
\overline\Gamma=1+2q.
\]
Hence, for every $\ell\ge1$,
\[
\sum_{k\ge\ell}w_k(\Gamma_k-\overline\Gamma)
=
\ell\,a\,q^\ell
\ge0.
\]
Thus Assumption~D holds along the whole interpolation.

For height, the GRD tangent is constant on the upper tail:
\[
\log\frac{u_{\tau_1}(0)}{u_{\tau_0}(0)}=0,
\qquad
\log\frac{u_{\tau_1}(k)}{u_{\tau_0}(k)}
=\log(\tau_1/\tau_0)\quad(k\ge1).
\]
Therefore Assumption~H reduces to monotonicity of the single tail $\nu_{\tau,\lambda}([1,\infty))$.  Directly from the same geometric formulas,
\[
\nu_{\tau,\lambda}([1,\infty))
=
\frac{\tau(1+\lambda)}{\lambda^2+2\tau\lambda+\tau},
\]
and
\[
\frac{d}{d\lambda}\nu_{\tau,\lambda}([1,\infty))
=
-\frac{\tau(\lambda^2+2\lambda+\tau)}
{(\lambda^2+2\tau\lambda+\tau)^2}
<0.
\]
Since the height optimiser satisfies $\lambda_\tau^\ast>\lambda_\tau$, Assumption~H follows for every pair in this family.  This example illustrates the intended use of the profile-order assumptions: once the sensitivity measure has the right tail order, GRD supplies the increasing tangent.
\end{remark}

The preceding examples separate the main phenomena.  Affine attachment gives a benchmark where the constants can be inspected directly, the bounded one-jump class gives a genuine verification of the profile-order assumptions, and the finite-dimensional counterexample shows why GRD alone is too weak.  The following remarks then place these examples in the broader regularity regimes covered by the analytic hypotheses.

\begin{remark}

The analytic conditions (I.2)--(I.4) are standard in CMJ theory and hold in the classical preferential attachment regimes below, but can fail for pathological or near-critical choices of $f$. The mean-depth identification in (I.5) and the height $L^1$ law in (I.6) are explicit probabilistic inputs and must be verified for the model under consideration.

\begin{enumerate}[label=(\roman*)]
\item \textbf{Bounded or sublinear attachment.}
If $f(k)\asymp k^\alpha$ with $0\le \alpha<1$, then
$\sum_{i< n}1/f(i)\asymp n^{1-\alpha}$, so the factors
$A_n(\lambda)=\prod_{i=0}^{n-1}\frac{f(i)}{f(i)+\lambda}$ decay exponentially when $\alpha=0$ and stretched-exponentially when $0<\alpha<1$.
In particular, $m_f(\lambda)<\infty$ for all $\lambda>0$, and the derivative condition
$0<-m_f'(\lambda_f)<\infty$ holds automatically.

\item \textbf{Asymptotically linear attachment.}
If $f(k)\sim ck$ with $c>0$, then $\sum_{i<n}1/f(i)\sim c^{-1}\log n$ and a logarithmic expansion gives
$A_n(\lambda)=n^{-\lambda/c+o(1)}$. Consequently $m_f(\lambda)=\infty$ for $\lambda<c$ and $m_f(\lambda)<\infty$ for $\lambda>c$; asymptotic linearity alone does not decide convergence at $\lambda=c$. Every Malthusian root therefore satisfies $\lambda_f\ge c$. If $\lambda_f>c$, then $\sum_n(\log n)A_n(\lambda_f)<\infty$ and hence $0<-m_f'(\lambda_f)<\infty$. The affine case $f(k)=k+\delta$ lies safely in this interior regime, since $\lambda_f=1+\delta>1=c$ by Remark~\ref{rem:affine-constants}.

\item \textbf{Too-fast growth.}
If $f(k)\gtrsim k\log k$, then $\sum_{i< n}1/f(i)\lesssim \log\log n$, so $A_n(\lambda)$ decays only like a power of $\log n$,
and in particular $m_f(\lambda)=\infty$ for all $\lambda>0$ in many such cases. For even faster growth one may also violate non-explosion.

\item \textbf{The critical linear boundary.}
At $\lambda=c$, lower-order terms determine both convergence and derivative integrability. For example, an eventually increasing choice
$f(k)=ck\bigl(1-2/\log k+O((\log k)^{-2})\bigr)$ gives
$A_n(c)\asymp [n(\log n)^2]^{-1}$. Thus $m_f(c)<\infty$ while $-m_f'(c)=\infty$; after tuning finitely many initial rates so that $m_f(c)=1$, the Malthusian root is the boundary point $\lambda_f=c$ and (I.4) fails. Hence asymptotic linearity by itself does not imply the finite-derivative condition.
\end{enumerate}
\end{remark}

\begin{remark}\label{rem:practical-sublinear-criteria}
Regular variation is not a standing assumption in the theorems; it is useful as a concrete way to verify the tail part of Assumption~H. For sublinear applications, suppose that, along the interpolation
$f_s=g(f/g)^s$, the functions $f_s$ are eventually nondecreasing and regularly varying with a common index $\rho\in[0,1)$, and that the following uniform estimates hold:
\[
  c_-(k+1)^\rho \le f_s(k)\le c_+(k+1)^\rho,
  \qquad
  |\log(f(k)/g(k))|\le C\log(k+2),
\]
with constants independent of $s\in[0,1]$. Then the usual product estimate
\[
  A_n^{(s)}(\lambda)
  \le
  \exp\left\{-c_\lambda n^{1-\rho}\right\},
  \qquad \lambda>0,
\]
holds uniformly in $s$ on compact $\lambda$--intervals bounded away from zero. Consequently (I.2)--(I.4) and the domination condition (II.2) follow from elementary summability estimates; the global monotonicity condition (I.1) is separate, and in the GRD setting it follows from the global monotonicity of $g$ and $h=f/g$. The same bounds are the standard input for the CMJ/BRW laws in (I.5)--(I.6). The remaining analytic condition (II.3) follows from standard envelope-theorem arguments whenever the height variational problem has a unique nondegenerate interior maximiser that stays in a compact $\lambda$--window along the interpolation.

Thus, in this regime, the only genuinely order-theoretic height condition is Assumption~H. Lemma~\ref{lem:rv-height} proves the required one-step monotonicity for all sufficiently large $k$; if the regular-variation estimates are uniform in $s$ (for instance through a uniform Potter bound), the tail threshold can be chosen uniformly. It remains only to rule out a finite initial ``cliff''. A convenient sufficient finite-prefix check is
\[
q_k^{(s)}(\lambda)\,\mathcal R_{k+1}^{(s)}(\lambda)
\ge
c_k^{(s)}(\lambda)-c_{k+1}^{(s)}(\lambda),
\qquad
0\le k<K,\quad
s\in[0,1],\quad
\lambda\in[\lambda_s,\lambda_s^\ast],
\]
where $K$ is any uniform tail threshold. If this finite-index family of inequalities holds, then
$k\mapsto\widetilde M_k^{(s)}(\lambda)$ is nondecreasing on all of $\N_0$, hence the sufficient tail criterion \eqref{eq:height-tail-sufficient} holds, and therefore Assumption~H follows.

These conditions cover the smooth sublinear examples one usually has in mind, such as shifted powers and slowly varying perturbations with no abrupt finite-prefix jumps. Abrupt jumps in the first few attachment weights are exactly the obstruction: the regular variation tail alone cannot control them.
\end{remark}

\section[Proof of depth monotonicity]{Proof of Theorem~\ref{thm:monotonicity}}
\label{sec:depth-proof}

\subsection{Interpolation reduction}

Fix increasing $f,g:\N_0\to(0,\infty)$ with $f\succeq_{\mathrm{GR}} g$, and set
\[
h(k):=\frac{f(k)}{g(k)},\qquad k\in\N_0.
\]
Then $h$ is nondecreasing. Define the interpolation
\begin{equation}\label{eq:f-theta}
f_\theta(k):=g(k)\,h(k)^\theta,\qquad \theta\in[0,1],\ k\in\N_0,
\end{equation}
so that $f_0=g$ and $f_1=f$.

For each $\theta\in[0,1]$, let $m_\theta(\lambda):=m_{f_\theta}(\lambda)$ be the Laplace transform from
\eqref{eq:mf-def} with representation \eqref{eq:mf-series}. By Assumption~II (II.1), there is a unique $\lambda_\theta>0$ such that
\begin{equation*}
m_\theta(\lambda_\theta)=1.
\end{equation*}
Assumption~II (II.2) justifies the differentiations and sum interchanges below; see Appendix~\ref{app:analytic-regularity}. No mean-depth identification is needed for the intermediate $f_\theta$: Assumption~I (I.5) is used only at $\theta=0,1$.

We further define
\begin{equation*}
Q_\theta:=-\lambda_\theta\,\partial_\lambda m_\theta(\lambda_\theta).
\end{equation*}
\begin{lemma}\label{lem:lambda-regularity}
Assume that Assumptions~II (II.1) and (II.2) hold.
Then the solution $\lambda_\theta$ to $m_\theta(\lambda_\theta)=1$ exists, is unique, and the map
$\theta\mapsto\lambda_\theta$ is continuously differentiable on $[0,1]$.
\end{lemma}

\begin{proof}
Existence and uniqueness of $\lambda_\theta$ follow from (II.1). The general tail identity from Appendix~\ref{app:analytic-regularity} and (II.2) give
\[
0<-\partial_\lambda m_\theta(\lambda_\theta)
=\sum_{i\ge0}\frac{r_i^{(\theta)}(\lambda_\theta)}{f_\theta(i)+\lambda_\theta}
\le \sum_{i\ge0}\frac{r_i^{(\theta)}(\lambda_0)}{f_\theta(i)+\lambda_0}<\infty.
\]
Condition (II.2) justifies termwise differentiation in $\theta$ and $\lambda$
and provides a uniform dominating envelope near $\lambda_\theta$.
Hence $(\theta,\lambda)\mapsto m_\theta(\lambda)$ is $C^1$ in a neighbourhood of the curve $\lambda=\lambda_\theta$,
and the implicit function theorem yields
\(
\lambda_\theta'=-\frac{\partial_\theta m_\theta(\lambda_\theta)}{\partial_\lambda m_\theta(\lambda_\theta)}.
\)
\end{proof}

Since $Q_0=Q_g$ and $Q_1=Q_f$, it remains to prove that $Q_\theta$ is nondecreasing. By the mean-value theorem, it is enough to show that $Q_\theta$ is continuous on $[0,1]$, differentiable on $(0,1)$, and $Q_\theta'\ge0$ there. A direct differentiation produces
terms that do not have a definite sign. The purpose of the \emph{gauge transformation} introduced next is to
choose a $\theta$-dependent global rescaling of $f_\theta$ which leaves $Q_\theta$ unchanged
(cf.~Remark~\ref{rem:normalise-lambda}), but eliminates the problematic global term in $Q_\theta'$.

\subsection{Gauge normalisation and additional notation}

A useful observation is that $Q_f$ is invariant under constant rescaling of $f$
(see Remark~\ref{rem:normalise-lambda}).
We exploit this freedom to choose a $\theta$-dependent rescaling (a ``gauge'') that removes the
global $\lambda_\theta'$ contribution when differentiating $Q_\theta$.

Let $f_\theta$ be the GRD interpolation \eqref{eq:f-theta} and let $\lambda_\theta$ be the unique solution of
$m_{f_\theta}(\lambda_\theta)=1$.
Define the gauged family
\begin{equation*}
f_\theta^*(k):=\frac{1}{\lambda_\theta}\,f_\theta(k),\qquad k\in\N_0.
\end{equation*}
By the scaling identity $m_{c f}(\lambda)=m_f(\lambda/c)$, this choice fixes the Malthusian parameter of $f_\theta^*$:
\begin{equation}\label{eq:malthus-fixed}
m_{f_\theta^*}(1)=m_{f_\theta}(\lambda_\theta)=1,\qquad \theta\in[0,1],
\end{equation}
and, moreover, $Q_{f_\theta^*}=Q_{f_\theta}$ by scaling invariance.

\medskip

\noindent\textit{Product weights and tails.}
For $\lambda>0$ define
\[
A_n^{(\theta,*)}(\lambda)
:=\prod_{i=0}^{n-1}\frac{f_\theta^*(i)}{f_\theta^*(i)+\lambda},
\qquad n\ge 0,\qquad (A_0^{(\theta,*)}(\lambda):=1),
\]
and the tail sums
\[
r_i^{(\theta,*)}(\lambda):=\sum_{n\ge i+1}A_n^{(\theta,*)}(\lambda),\qquad i\ge 0.
\]
We will evaluate at $\lambda=1$ and suppress the argument:
\[
A_n^{(\theta,*)}:=A_n^{(\theta,*)}(1),\qquad r_i^{(\theta,*)}:=r_i^{(\theta,*)}(1).
\]
Since $m_{f_\theta^*}(1)=\sum_{n\ge1}A_n^{(\theta,*)}=1$, the sequence $(A_n^{(\theta,*)})_{n\ge1}$ is a probability mass function;
let $N_\theta$ be the $\N$-valued random variable with $\PP(N_\theta=n)=A_n^{(\theta,*)}$.

Define the basic coefficients and weights
\[
\alpha_i^{(\theta)}:=\frac{1}{f_\theta^*(i)+1},\qquad
w_i^{(\theta)}:=\alpha_i^{(\theta)}\,r_i^{(\theta,*)},\qquad i\ge 0.
\]
The tail representation \eqref{eq:mfprime-tail} gives
\begin{equation*}
Q_\theta=Q_{f_\theta^*}=\sum_{i\ge 0} w_i^{(\theta)}.
\end{equation*}

\noindent\textit{Log-derivative profile and centering.}
Set
\begin{equation}\label{eq:b-i-def}
b_i^{(\theta)}:=\log h(i)+a'(\theta),\qquad i\in\N_0,
\end{equation}
where $a(\theta):=-\log\lambda_\theta$ so that $f_\theta^*=e^{a(\theta)}f_\theta$.
Since $h$ is nondecreasing, the map $i\mapsto b_i^{(\theta)}$ is nondecreasing for each fixed $\theta$.

For $\theta\in(0,1)$, differentiating \eqref{eq:malthus-fixed} yields $\partial_\theta m_{f_\theta^*}(1)=0$.
Lemma~\ref{lem:dtheta-m-tail} below rewrites this as the centering condition
\begin{equation}\label{eq:centering}
\sum_{i\ge0} b_i^{(\theta)}\, w_i^{(\theta)}=0.
\end{equation}
Equivalently,
\begin{equation}\label{eq:gauge-choice}
a'(\theta)=-\frac{u_\theta}{Q_\theta},
\qquad
u_\theta:=\sum_{i\ge0}(\log h(i))\,w_i^{(\theta)}.
\end{equation}
This choice removes the global derivative term and is the crucial input for calculating the sign of $Q_\theta'$.
\subsection{Differentiation identities}

We record the basic $\theta$-derivatives of the product terms $A_n^{(\theta,*)}(\lambda)$ and of $m_\theta^*(\lambda)$.
All interchanges of differentiation and summation are justified by our standing analytic assumptions.

\begin{lemma}
For each $n\ge 1$ and $\lambda>0$,
\begin{equation}\label{eq:dtheta-A}
\partial_\theta A_n^{(\theta,*)}(\lambda)
=
A_n^{(\theta,*)}(\lambda)\sum_{i=0}^{n-1}\frac{\lambda}{f_\theta^*(i)+\lambda}\,b_i^{(\theta)},
\end{equation}
where $b_i^{(\theta)}$ is defined in \eqref{eq:b-i-def}.
\end{lemma}

\begin{proof}
For each $i$,
\[
\partial_\theta \log\Big(\frac{f_\theta^*(i)}{f_\theta^*(i)+\lambda}\Big)
=
\partial_\theta\log f_\theta^*(i)-\partial_\theta\log(f_\theta^*(i)+\lambda)
=
b_i^{(\theta)}-\frac{f_\theta^*(i)}{f_\theta^*(i)+\lambda}\,b_i^{(\theta)}
=
\frac{\lambda}{f_\theta^*(i)+\lambda}\,b_i^{(\theta)}.
\]
Summing over $i=0,\dots,n-1$ and multiplying by $A_n^{(\theta,*)}(\lambda)$ yields \eqref{eq:dtheta-A}.
\end{proof}

\begin{lemma}\label{lem:dtheta-m-tail}
Under Assumptions~II (II.1)--(II.2), there is an open interval $I$ containing $1$ such that, for every
$\theta\in(0,1)$ and $\lambda\in I$,
\begin{equation}\label{eq:dtheta-m-tail}
\partial_\theta m_\theta^*(\lambda)
=
\sum_{i\ge 0}\frac{\lambda}{f_\theta^*(i)+\lambda}\,b_i^{(\theta)}\,r_i^{(\theta,*)}(\lambda).
\end{equation}
\end{lemma}

\begin{proof}
Let $\bar\lambda=(\lambda_0+\lambda_-)/2$, where $\lambda_0$ and $\lambda_-$ are from (II.2) and (II.1), respectively.
Choose $\eta\in(0,1)$ so that $\lambda_-(1-\eta)>\bar\lambda$, and set $I=(1-\eta,1+\eta)$.
Fix $\theta\in(0,1)$ and $\lambda\in I$.
We start from the definition
\[
m_\theta^*(\lambda)=\sum_{n\ge 1}A_n^{(\theta,*)}(\lambda),
\qquad
A_n^{(\theta,*)}(\lambda):=\prod_{j=0}^{n-1}\frac{f_\theta^*(j)}{f_\theta^*(j)+\lambda}.
\]
We justify interchanging $\partial_\theta$ with the infinite sum by dominated convergence. For
$\vartheta\in[0,1]$ put $\mu_\vartheta:=\lambda_\vartheta\lambda$. Then
$\mu_\vartheta\ge\lambda_-(1-\eta)>\bar\lambda$, and
\[
A_n^{(\vartheta,*)}(\lambda)=A_n^{(\vartheta)}(\mu_\vartheta),
\qquad
\frac{\lambda}{f_\vartheta^*(j)+\lambda}
=
\frac{\mu_\vartheta}{f_\vartheta(j)+\mu_\vartheta}.
\]
The exponential-gap estimate in Appendix~\ref{app:analytic-regularity}, together with the boundedness of
$\lambda_\vartheta$ and $a'(\vartheta)$ from Lemma~\ref{lem:lambda-regularity}, therefore gives a summable envelope, uniform for $\vartheta$ near $\theta$, for
\[
A_n^{(\vartheta,*)}(\lambda)
\sum_{j=0}^{n-1}
\frac{\lambda\,|b_j^{(\vartheta)}|}{f_\vartheta^*(j)+\lambda}.
\]
Indeed, $|b_j^{(\vartheta)}|\le |\log h(j)|+\sup_{u\in[0,1]}|a'(u)|$, so this is precisely the first-derivative envelope furnished by (II.2) and the gap estimate.
By definition of $b_j^{(\theta)}=\partial_\theta\log f_\theta^*(j)$ we have
\[
\partial_\theta A_n^{(\theta,*)}(\lambda)
=
A_n^{(\theta,*)}(\lambda)\sum_{j=0}^{n-1}\frac{\lambda}{f_\theta^*(j)+\lambda}\,b_j^{(\theta)}.
\]
Therefore,
\[
\partial_\theta m_\theta^*(\lambda)
=
\sum_{n\ge 1}\partial_\theta A_n^{(\theta,*)}(\lambda).
\]

Now expand the derivative and rearrange the resulting double sum:
\[
\sum_{n\ge 1}\partial_\theta A_n^{(\theta,*)}(\lambda)
=
\sum_{n\ge 1}A_n^{(\theta,*)}(\lambda)
\sum_{j=0}^{n-1}\frac{\lambda}{f_\theta^*(j)+\lambda}\,b_j^{(\theta)}
=
\sum_{j\ge 0}\frac{\lambda}{f_\theta^*(j)+\lambda}\,b_j^{(\theta)}
\sum_{n\ge j+1}A_n^{(\theta,*)}(\lambda),
\]
where the exchange of the order of summation is justified by the same absolute-summability bound. Finally, the inner tail equals $r_j^{(\theta,*)}(\lambda)$ by definition,
giving \eqref{eq:dtheta-m-tail}.
\end{proof}

\subsection[Asymptotic expansion of Qtheta]{Asymptotic expansion of \texorpdfstring{$Q_{\theta}$}{Qtheta} and its tail expansion}

We now compute $Q_\theta'$. Recall that $Q_\theta=Q_{f_\theta}=Q_{f_\theta^*}$ by scaling invariance, so we differentiate $Q_\theta$ along the gauged family.

\begin{lemma}\label{lem:Qprime}
The map $\theta\mapsto Q_\theta$ is continuous on $[0,1]$ and differentiable on $(0,1)$. With the gauge choice \eqref{eq:gauge-choice}, its derivative satisfies
\begin{equation}\label{eq:Qprime-double}
Q_\theta'
= \sum_{0\le i\le j}\big(b_i^{(\theta)}+b_j^{(\theta)}\big)\,\alpha_i^{(\theta)}\,\alpha_j^{(\theta)}\,r_j^{(\theta,*)},
\end{equation}
where $\alpha_i^{(\theta)}=(f_\theta^*(i)+1)^{-1}$ and $r_j^{(\theta,*)}=r_j^{(\theta,*)}(1)$.
\end{lemma}

\begin{proof}
The domination proved in Appendix~\ref{app:analytic-regularity} first gives continuity of $\partial_\lambda m_\theta^*(1)$ in $\theta$, hence of $Q_\theta=-\partial_\lambda m_\theta^*(1)$. For $\theta\in(0,1)$ it also justifies the mixed derivative used below.
Along the gauged family, $m_{f_\theta^*}(1)=1$ for all $\theta$ by \eqref{eq:malthus-fixed}.
Thus $Q_\theta=-\partial_\lambda m_\theta^*(1)$ and
\[
Q_\theta'=-\partial_{\theta\lambda} m_\theta^*(1).
\]
Using \eqref{eq:dtheta-m-tail} and differentiating in $\lambda$ at $\lambda=1$ gives
\[
\partial_{\theta\lambda}m_\theta^*(1)
=
\sum_{i\ge0} b_i^{(\theta)}\,
\partial_\lambda\Big(\frac{\lambda}{f_\theta^*(i)+\lambda}\,r_i^{(\theta,*)}(\lambda)\Big)\Big|_{\lambda=1}.
\]

Expanding the $\lambda$-derivative gives
\[
\partial_{\theta\lambda} m_\theta^*(\lambda)
=\sum_{i\ge0} b_i^{(\theta)}\,\partial_\lambda\Big(\frac{\lambda}{f_\theta^*(i)+\lambda}\,r_i^{(\theta,*)}(\lambda)\Big),
\qquad
r_i^{(\theta,*)}(\lambda):=\sum_{n\ge i+1}A_n^{(\theta,*)}(\lambda),
\]
where $A_n^{(\theta,*)}(\lambda)=\prod_{j=0}^{n-1}\frac{f_\theta^*(j)}{f_\theta^*(j)+\lambda}$.
Writing $\alpha_i(\lambda):=(f_\theta^*(i)+\lambda)^{-1}$, we have
\[
\partial_\lambda\Big(\lambda \alpha_i(\lambda) r_i(\lambda)\Big)
=\big(\alpha_i(\lambda)-\lambda\alpha_i(\lambda)^2\big)\,r_i(\lambda)
+\lambda\alpha_i(\lambda)\,\partial_\lambda r_i(\lambda).
\]
Moreover, differentiating termwise and using
\(
\partial_\lambda A_n(\lambda)=-A_n(\lambda)\sum_{j=0}^{n-1}\alpha_j(\lambda)
\)
yields the useful identity
\begin{equation}\label{eq:dlam-ri}
\partial_\lambda r_i(\lambda)
=-\sum_{j\ge0}\alpha_j(\lambda)\,r_{\max\{i,j\}}(\lambda),
\end{equation}
Indeed,
\begin{align*}
\partial_\lambda r_i(\lambda)
&=-\sum_{n\ge i+1}A_n(\lambda)\sum_{j=0}^{n-1}\alpha_j(\lambda)\\
&=-\sum_{j\ge 0}\alpha_j(\lambda)
  \sum_{n\ge \max\{i+1,j+1\}}A_n(\lambda)\\
&=-\sum_{j\ge 0}\alpha_j(\lambda)\,r_{\max\{i,j\}}(\lambda),
\end{align*}
which is \eqref{eq:dlam-ri}. Evaluating at $\lambda=1$ and writing
\(
\alpha_i:=\alpha_i(1)=(f_\theta^*(i)+1)^{-1}
\)
and
\(
r_i:=r_i^{(\theta,*)}(1),
\)
we obtain
\[
Q_\theta'=-\partial_{\theta\lambda}m_\theta^*(1)
=-\sum_{i\ge 0} b_i^{(\theta)}\Big((\alpha_i-\alpha_i^2)\,r_i+\alpha_i\,\partial_\lambda r_i(\lambda)\big|_{\lambda=1}\Big).
\]
The centering identity \eqref{eq:centering} (equivalently $\sum_{i\ge 0} b_i^{(\theta)}\alpha_i r_i=0$) cancels the linear term
$-\sum_i b_i^{(\theta)}\alpha_i r_i$, and \eqref{eq:dlam-ri} yields
\[
Q_\theta'
=\sum_{i\ge 0} b_i^{(\theta)}\alpha_i^2 r_i
+\sum_{i,j\ge 0} b_i^{(\theta)}\alpha_i\alpha_j\,r_{\max\{i,j\}}.
\]
Splitting the last double sum into the regions $i<j$, $i=j$, and $i>j$, and relabeling $(i,j)\mapsto (j,i)$ in the region $i>j$, we obtain
\[
Q_\theta'=\sum_{0\le i\le j}\big(b_i^{(\theta)}+b_j^{(\theta)}\big)\,\alpha_i\alpha_j\,r_j,
\]
which is \eqref{eq:Qprime-double}.
\end{proof}

\subsection[Conclusion: Qtheta derivative nonnegative]{Conclusion: \texorpdfstring{$Q_{\theta}'\ge 0$}{Qtheta derivative nonnegative}}
\label{subsec:Qprime-nonneg}

We now complete the proof that $Q_\theta'\ge 0$ along the interpolation, using the representation from
Lemma~\ref{lem:Qprime} and the gauge centering $\sum_{k\ge 0} b_k^{(\theta)} w_k^{(\theta)}=0$.

Throughout this subsection we fix $\theta$ and suppress the superscripts $(\theta)$ and $(\theta,*)$.
Recall the notation
\[
\alpha_k:=\frac{1}{f_\theta^*(k)+1},
\qquad
r_k:=r_k^{(\theta,*)}(1),
\qquad
w_k:=\alpha_k r_k,
\qquad
b_k:=\log h(k)+a'(\theta).
\]
We also define $N_\theta$ to be a generic random variable with $\PP(N_\theta=n)=A_n^{(\theta,*)}$, so that $r_k=\PP(N_\theta>k)$ and $w_k=\PP(N_\theta>k)/(f_\theta^*(k)+1)$.
It holds that $k\mapsto b_k$ is nondecreasing and the gauge implies
\begin{equation}
\label{eq:center-bw}
\sum_{k\ge 0} b_k w_k=0.
\end{equation}

\noindent\textit{The corrected regression profile.}
Define, for $k\ge 0$,
\begin{equation*}
C_k
:=\frac{1}{r_k}\sum_{j\ge 0}\alpha_j\, r_{\max\{k,j\}}.
\end{equation*}
The profile that occurs after collecting the coefficients of $b_k$ is
\begin{equation*}
\Gamma_k:=C_k+\alpha_k.
\end{equation*}
Equivalently, writing
\[
W_k:=\sum_{j\ge k}\alpha_j r_j,
\qquad
s_k:=\sum_{j=0}^{k-1}\alpha_j\quad(s_0:=0),
\]
we have the useful identity
\begin{equation}
\label{eq:Ck-alt}
C_k=s_k+\frac{W_k}{r_k}.
\end{equation}

\begin{lemma}
\label{lem:Ck-monotone}
The sequence $k\mapsto C_k$ is nondecreasing.
\end{lemma}

\begin{proof}
From \eqref{eq:Ck-alt} and $W_{k+1}=W_k-\alpha_k r_k$ we compute
\[
C_{k+1}-C_k
=\alpha_k+\frac{W_{k+1}}{r_{k+1}}-\frac{W_k}{r_k}
=\alpha_k+\frac{W_k-\alpha_k r_k}{r_{k+1}}-\frac{W_k}{r_k}
=\frac{r_k-r_{k+1}}{r_{k+1}}\Big(\frac{W_k}{r_k}-\alpha_k\Big).
\]
Since $r_k-r_{k+1}=\PP(N_\theta=k+1)\ge 0$, it remains to show $W_k/r_k\ge \alpha_k$.
But
\[
\frac{W_k}{r_k}
=\frac{1}{\PP(N_\theta>k)}\sum_{j\ge k}\alpha_j\,\PP(N_\theta>j)
=\E\left[\sum_{j=k}^{N_\theta-1}\alpha_j \,\middle|\, N_\theta>k\right]
\ge \E\left[\alpha_k\,\middle|\, N_\theta>k\right]
=\alpha_k,
\]
and therefore $C_{k+1}\ge C_k$.
\end{proof}

The preceding lemma is useful but not sufficient by itself: since $f_\theta^\ast$ is nondecreasing, $\alpha_k$ is nonincreasing, so $\Gamma_k=C_k+\alpha_k$ need not be monotone. From the proof of Lemma~\ref{lem:Ck-monotone} one also obtains
\[
C_{k+1}-C_k
=
\frac{(r_k-r_{k+1})W_{k+1}}{r_kr_{k+1}},
\]
and hence the one-step condition
\begin{equation}\label{eq:depth-onestep-sufficient}
\frac{(r_k-r_{k+1})W_{k+1}}{r_kr_{k+1}}
\ge
\alpha_k-\alpha_{k+1},
\qquad k\ge0,
\end{equation}
is a convenient sufficient condition for pointwise monotonicity of $\Gamma_k$.

Starting from the double-sum representation in Lemma~\ref{lem:Qprime} we rewrite
\begin{align}
Q_\theta'
&=\sum_{0\le i\le j}(b_i+b_j)\,\alpha_i\alpha_j\,r_j \notag\\
&=\sum_{k\ge 0} b_k\,\alpha_k
\Big(\alpha_k r_k+\sum_{j\ge 0}\alpha_j\,r_{\max\{k,j\}}\Big) \notag\\
&=\sum_{k\ge 0} b_k\,(\alpha_k r_k)\,(C_k+\alpha_k)
=\sum_{k\ge 0} b_k\,w_k\,\Gamma_k.
\label{eq:Qprime-Gamma}
\end{align}
The additional term $\sum_k b_k\alpha_k^2r_k$ is the diagonal contribution in \eqref{eq:Qprime-double}.

Let
\begin{equation}\label{eq:Gamma-bar-depth}
\overline\Gamma:=
\frac{\sum_{k\ge0}w_k\Gamma_k}{\sum_{k\ge0}w_k},
\qquad
S_\ell:=
\sum_{k\ge\ell}w_k(\Gamma_k-\overline\Gamma),
\qquad \ell\ge1.
\end{equation}
By the centering identity \eqref{eq:center-bw},
\[
Q_\theta'
=
\sum_{k\ge0}b_kw_k(\Gamma_k-\overline\Gamma).
\]
The identity may be read first with the sums truncated at $k\le M$; the domination in (II.2) and the finiteness required in Assumption~D allow the limit $M\to\infty$.
Writing $\Delta b_\ell:=b_\ell-b_{\ell-1}=\log h(\ell)-\log h(\ell-1)$ and using
$b_k=b_0+\sum_{\ell=1}^k\Delta b_\ell$, Abel summation gives the corrected core identity
\begin{equation}\label{eq:Qprime-Abel-depth}
Q_\theta'
=
\sum_{\ell\ge1}
\big(\log h(\ell)-\log h(\ell-1)\big)S_\ell.
\end{equation}

We also record the weighted Chebyshev inequality used below in the height proof and in sufficient-condition checks.

\begin{lemma}\label{lem:weighted-chebyshev-eventual}
Let $(w_k)_{k\ge0}$ be nonnegative weights with $0<W:=\sum_{k\ge0}w_k<\infty$.
Let $(a_k)_{k\ge0}$ and $(c_k)_{k\ge0}$ be nondecreasing real sequences such that the sums below are finite.
If
\[
\sum_{k\ge0}a_k w_k=0,
\]
then
\[
\sum_{k\ge0}a_k c_k w_k\ge0.
\]
\end{lemma}

\begin{proof}
Normalize the weights by $\nu_k=w_k/W$. Since $a$ and $c$ are nondecreasing,
\[
\sum_{k\ge0}a_kc_k\nu_k
-\Big(\sum_{k\ge0}a_k\nu_k\Big)\Big(\sum_{k\ge0}c_k\nu_k\Big)
=\frac12\sum_{i,j\ge0}(a_i-a_j)(c_i-c_j)\nu_i\nu_j\ge0.
\]
The centering assumption gives the claim.
\end{proof}

\noindent\textit{The sign of $Q_\theta'$.}
The Abel identity \eqref{eq:Qprime-Abel-depth} is the concrete form of the dual pairing from Lemma~\ref{lem:increasing-dual-cone}.  Under GRD the tangent $\log h$ is nondecreasing, while Assumption~D says that the signed measure $w_k(\Gamma_k-\overline\Gamma)$ has nonnegative upper tails.  Hence $Q_\theta'\ge0$.

\begin{proof}[Proof of Theorem~\ref{thm:monotonicity}]
By Lemma~\ref{lem:Qprime}, the corrected profile representation \eqref{eq:Qprime-Gamma}, and Assumption~D through the dual-cone identity \eqref{eq:Qprime-Abel-depth}, we have $Q_\theta'\ge 0$ for every $\theta\in(0,1)$.
Lemma~\ref{lem:Qprime} and the mean-value theorem therefore show that $\theta\mapsto Q_\theta$ is nondecreasing along \eqref{eq:f-theta}, so $Q_1\ge Q_0$, i.e.\ $Q_f\ge Q_g$.
The endpoint identifications $\mathsf{c}_f=1/Q_f$ and $\mathsf{c}_g=1/Q_g$ from Assumption~I (I.5) then give $\mathsf{c}_f\le \mathsf{c}_g$, which is the claimed monotonicity of the depth constant.
\end{proof}

\section[Proof of height monotonicity]{Proof of Theorem~\ref{thm:height-monotonicity}}
\label{sec:height-proof}

This section proves Theorem~\ref{thm:height-monotonicity} via the multiplicative interpolation from the theorem statement.
The CMJ/BRW embedding and the identification of the height constant with $R_f:=\lambda_f\kappa_f$ are recalled in Section~\ref{sec:cmj-parameters}.
We now focus on the monotonicity of $R_s$ along the GRD interpolation.

\subsection{Interpolation identity for \texorpdfstring{$R_s$}{Rs}}

Fix attachment functions $g,f$ with $f\succeq_{\mathrm{GR}} g$, and write $h:=f/g$, so that $k\mapsto h(k)$ is nondecreasing.
As in the statement of Theorem~\ref{thm:height-monotonicity}, we interpolate
\begin{equation}\label{eq:height-interpolation}
	f_s(k)\;:=\;g(k)\,h(k)^s,\qquad s\in[0,1],
\end{equation}
Set
\begin{equation*}
	b_k=b_k^{(s)}\;:=\;\partial_s\log f_s(k)\;=\;\log h(k),
\end{equation*}
so $k\mapsto b_k$ is nondecreasing for each $s$.

Let $\lambda_s$ be the Malthusian parameter of the CMJ process associated with $f_s$, and let $\kappa_s$ be the BRW speed from Lemma~\ref{lem:height-variational}.
Multiplying an attachment function by a scalar does not change the \emph{discrete} preferential-attachment law, while it rescales $\lambda_s$ and $\kappa_s$ inversely. The product
\begin{equation*}
	R_s\;:=\;\lambda_s\,\kappa_s
\end{equation*}
is therefore invariant under scalar rescaling.
In particular, the height constant can be written as
\begin{equation*}
	\mathsf{c}^{\ast}_{f_s}=\frac{1}{R_s}.
\end{equation*}
Thus Theorem~\ref{thm:height-monotonicity} reduces to proving that $s\mapsto R_s$ is nondecreasing.

Fix $s$ and $\lambda\ge\lambda_s$. Define
\begin{equation}\label{eq:An-ms-height}
	A_n^{(s)}(\lambda)\;:=\;\prod_{j=0}^{n-1}\frac{f_s(j)}{f_s(j)+\lambda},
	\qquad
	m_s(\lambda)\;:=\;\sum_{n\ge1}A_n^{(s)}(\lambda).
\end{equation}
Let $N$ be the associated ``root law'' under $\PP_{s,\lambda}$,
\begin{equation}\label{eq:rootlaw-height}
	\PP_{s,\lambda}(N=n)\;:=\;\frac{A_n^{(s)}(\lambda)}{m_s(\lambda)},\qquad n\ge1,
\end{equation}
with tail indicators $X_k:=\mathbf 1_{\{N>k\}}$ and tails $T_k:=\PP_{s,\lambda}(N>k)$.

Introduce the basic factors
\begin{equation*}
	c_k=c_k^{(s)}(\lambda)\;:=\;\frac{1}{f_s(k)+\lambda},
	\qquad
	a_k=a_k^{(s)}(\lambda)\;:=\;\frac{\lambda}{f_s(k)+\lambda}=\lambda\,c_k,
\end{equation*}
and the increasing functional
\begin{equation*}
	S\;:=\;\sum_{j=0}^{N-1}c_j.
\end{equation*}
Finally define the $s$--score random variable
\begin{equation}\label{eq:U-def-height}
	U\;:=\;\partial_s\log A_N^{(s)}(\lambda)
	\;=\;\sum_{k\ge0} b_k\,a_k\,X_k.
\end{equation}

\begin{lemma}\label{lem:score-identities-height}
	For each fixed $s$ and $\lambda\ge\lambda_s$ at which the displayed derivatives are finite,
	\begin{align}
		\partial_s\log m_s(\lambda) &= \E_{s,\lambda}[U], \label{eq:dslogm-height}\\
		\partial_\lambda\log m_s(\lambda) &= -\,\E_{s,\lambda}[S]. \label{eq:dlambda-logm-height}
	\end{align}
	Moreover, for any fixed integrable function $Y=Y(N)$ whose values do not depend on $s$,
	\begin{equation}\label{eq:score-rule-height}
		\partial_s \E_{s,\lambda}[Y]\;=\;\Cov_{s,\lambda}(Y,U).
	\end{equation}
\end{lemma}

\begin{proof}
	Differentiate term-by-term in \eqref{eq:An-ms-height} (justified by the analytic assumptions already used in the depth proof) to obtain
	$\partial_s A_n^{(s)}(\lambda)=A_n^{(s)}(\lambda)\,\partial_s\log A_n^{(s)}(\lambda)$ and hence
	$\partial_s m_s(\lambda)=m_s(\lambda)\E_{s,\lambda}[\partial_s\log A_N^{(s)}(\lambda)]=m_s(\lambda)\E_{s,\lambda}[U]$,
	which gives \eqref{eq:dslogm-height}.
	Similarly, $\partial_\lambda\log A_n^{(s)}(\lambda)=-\sum_{j=0}^{n-1}c_j$, yielding \eqref{eq:dlambda-logm-height}.
	Finally, for fixed $Y=Y(N)$, \eqref{eq:score-rule-height} is the standard likelihood identity under the tilted family \eqref{eq:rootlaw-height}.
\end{proof}

Define now the weights
\begin{equation*}
	w_k=w_k^{(s)}(\lambda)\;:=\;a_k\,T_k,
\end{equation*}
and the conditional shift
\begin{equation}\label{eq:Mk-def-height}
	M_k=M_k^{(s)}(\lambda)\;:=\;\E_{s,\lambda}[S\mid N>k]-\E_{s,\lambda}[S].
\end{equation}
Then expanding $\Cov(S,U)$ using \eqref{eq:U-def-height} gives the tail form
\begin{equation*}
	\Cov_{s,\lambda}(S,U)
	\;=\;
	\sum_{k\ge0} b_k\,w_k\,{M}_k.
\end{equation*}

\medskip
\noindent
\textit{Optimiser and the corrected shift.}
Let $\lambda^{\ast}_s\in(\lambda_s,\infty)$ denote an interior maximiser of the variational problem
\begin{equation}\label{eq:lamstar-def-height}
	\kappa_s=\sup_{\lambda>\lambda_s}\frac{-\log m_s(\lambda)}{\lambda}
	\;=\;\frac{-\log m_s(\lambda^{\ast}_s)}{\lambda^{\ast}_s}.
\end{equation}
By Lemma~\ref{lem:lamstar-characterisation}, at $\lambda=\lambda^{\ast}_s$ we also have
\begin{equation}\label{eq:kappa-as-ES-height}
	\kappa_s
	\;=\;
	-\partial_\lambda\log m_s(\lambda^{\ast}_s)
	\;=\;
	\E_{s,\lambda^{\ast}_s}[S].
\end{equation}
Define the local correction
\begin{equation*}
	d_k=d_k^{(s)}(\lambda)\;:=\;\frac{f_s(k)}{\lambda(f_s(k)+\lambda)}
	\;=\;\frac{1}{\lambda}-c_k,
\end{equation*}
and the corrected sequence
\begin{equation*}
	\widetilde M_k=\widetilde M_k^{(s)}(\lambda)
	\;:=\;
	M_k-d_k.
\end{equation*}

\begin{lemma}\label{lem:gamma-derivative-height}
	Assume that $s$ is a point at which $\lambda_s$ and $\kappa_s$ are differentiable, that the envelope identity in (II.3) holds, and that the analytic interchanges used in Lemma~\ref{lem:score-identities-height} apply at $\lambda=\lambda_s$ and $\lambda=\lambda^{\ast}_s$.
	Then
	\begin{equation}\label{eq:gamma-derivative-twopoint}
		R_s'
		=
		\kappa_s\,\frac{\E_{s,\lambda_s}[U]}{\E_{s,\lambda_s}[S]}
		\;-\;
		\frac{\lambda_s}{\lambda^{\ast}_s}\,\E_{s,\lambda^{\ast}_s}[U].
	\end{equation}
	Moreover, writing the $\lambda$--weighted mean
	\begin{equation}\label{eq:bbar-def-height}
		\bar b_s(\lambda)
		:=
		\frac{\sum_{k\ge0} b_k^{(s)}\,w_k^{(s)}(\lambda)}{\sum_{k\ge0} w_k^{(s)}(\lambda)},
		\qquad \lambda\ge\lambda_s,
	\end{equation}
	we have the weighted-mean form
	\begin{equation}\label{eq:gamma-derivative-bbar-height}
		R_s'
		=
		R_s\big(\bar b_s(\lambda_s)-\bar b_s(\lambda^{\ast}_s)\big).
	\end{equation}
\end{lemma}

\begin{proof}
	Fix $s\in(0,1)$.
	The Malthusian parameter is characterised by $m_s(\lambda_s)=1$, i.e.\ $\log m_s(\lambda_s)=0$.
	Differentiating in $s$ yields
	\[
	0
	=
	\partial_s\log m_s(\lambda_s)
	+
	\lambda_s'\,\partial_\lambda\log m_s(\lambda_s).
	\]
	By Lemma~\ref{lem:score-identities-height},
	$\partial_s\log m_s(\lambda)=\E_{s,\lambda}[U]$ and $\partial_\lambda\log m_s(\lambda)=-\E_{s,\lambda}[S]$,
	so
	\begin{equation*}
		\lambda_s'=\frac{\E_{s,\lambda_s}[U]}{\E_{s,\lambda_s}[S]}.
	\end{equation*}

	Let $J_s(\lambda):=-(\log m_s(\lambda))/\lambda$, $\lambda>\lambda_s$.
	By the envelope identity in Assumptions~II (II.3),
	\[
	\kappa_s'= \partial_s J_s(\lambda^{\ast}_s)
	=
	-\frac{1}{\lambda^{\ast}_s}\,\partial_s\log m_s(\lambda^{\ast}_s)
	=
	-\frac{1}{\lambda^{\ast}_s}\,\E_{s,\lambda^{\ast}_s}[U],
	\]
	using Lemma~\ref{lem:score-identities-height} once more.

	Combining the product rule $R_s'=\lambda_s'\kappa_s+\lambda_s\kappa_s'$ with the identities above yields
	\eqref{eq:gamma-derivative-twopoint}.

	Since $U=\sum_{k\ge0} b_k^{(s)}a_kX_k$ and $\E[X_k]=T_k$, we have
	\[
	\E_{s,\lambda}[U]=\sum_{k\ge0} b_k^{(s)}\,w_k^{(s)}(\lambda).
	\]
	Moreover $S=\sum_{k\ge0}c_kX_k$ implies $\E_{s,\lambda}[S]=\sum_{k\ge0}c_kT_k=(1/\lambda)\sum_{k\ge0}w_k^{(s)}(\lambda)$ because $w_k=\lambda c_kT_k$.
	Hence
	\[
	\frac{\E_{s,\lambda_s}[U]}{\E_{s,\lambda_s}[S]}
	=
	\lambda_s\,\bar b_s(\lambda_s).
	\]
	Finally, by \eqref{eq:kappa-as-ES-height} we have $\E_{s,\lambda^{\ast}_s}[S]=\kappa_s$, so
	\[
	\frac{\lambda_s}{\lambda^{\ast}_s}\,\E_{s,\lambda^{\ast}_s}[U]
	=
	\frac{\lambda_s}{\lambda^{\ast}_s}\,\bar b_s(\lambda^{\ast}_s)\sum_{k\ge0}w_k^{(s)}(\lambda^{\ast}_s)
	=
	\lambda_s\kappa_s\,\bar b_s(\lambda^{\ast}_s)
	=
	R_s\,\bar b_s(\lambda^{\ast}_s).
	\]
	Substituting these identities into \eqref{eq:gamma-derivative-twopoint} yields \eqref{eq:gamma-derivative-bbar-height}.
\end{proof}

By \eqref{eq:gamma-derivative-bbar-height}, the sign of $R_s'$ is governed by the comparison of the weighted means
$\bar b_s(\lambda_s)$ and $\bar b_s(\lambda^{\ast}_s)$. The following lemma expresses $\partial_\lambda \bar b_s(\lambda)$ in terms of the corrected profile $\widetilde M_k^{(s)}(\lambda)$.

\begin{lemma}\label{lem:bbar-derivative-height}
	Fix $s$ and write $W_s(\lambda):=\sum_{j\ge0} w_j^{(s)}(\lambda)$.
	Then $\lambda\mapsto \bar b_s(\lambda)$ is continuous on $[\lambda_s,\infty)$, differentiable on $(\lambda_s,\infty)$, and satisfies
	\begin{equation}\label{eq:bbar-derivative-height}
		\partial_\lambda \bar b_s(\lambda)
		=
		-\frac{1}{W_s(\lambda)}\sum_{k\ge0}\bigl(b_k^{(s)}-\bar b_s(\lambda)\bigr)\,w_k^{(s)}(\lambda)\,\widetilde M_k^{(s)}(\lambda).
	\end{equation}
	In particular, if $k\mapsto b_k^{(s)}$ and $k\mapsto \widetilde M_k^{(s)}(\lambda)$ are nondecreasing, then
	$\partial_\lambda \bar b_s(\lambda)\le 0$.

\end{lemma}

\begin{proof}
	Write
	\[
	B_s(\lambda):=\sum_{k\ge0} b_k^{(s)}w_k^{(s)}(\lambda),
	\qquad
	W_s(\lambda):=\sum_{k\ge0} w_k^{(s)}(\lambda),
	\]
	so that $\bar b_s(\lambda)=B_s(\lambda)/W_s(\lambda)$.
	The domination in (II.2) gives continuity of these two sums on $[\lambda_s,\infty)$.
	Differentiating the ratio yields
	\begin{equation}\label{eq:bbar-ratio-derivative}
		\partial_\lambda \bar b_s(\lambda)
		=
		\frac{1}{W_s(\lambda)}\sum_{k\ge0}\bigl(b_k^{(s)}-\bar b_s(\lambda)\bigr)\,\partial_\lambda w_k^{(s)}(\lambda).
	\end{equation}
	It remains to compute $\partial_\lambda w_k^{(s)}(\lambda)$. Since
	$w_k=a_kT_k$ with $a_k=\lambda/(f_s(k)+\lambda)$ and $T_k=\PP_{s,\lambda}(N>k)$, we have
	\[
	\partial_\lambda a_k
	=
	\frac{f_s(k)}{(f_s(k)+\lambda)^2}
	=
	a_k\,d_k,
	\qquad
	d_k=\frac{f_s(k)}{\lambda(f_s(k)+\lambda)}=\frac{1}{\lambda}-c_k.
	\]
	Moreover, differentiating $\log \PP_{s,\lambda}(N=n)=\log A_n^{(s)}(\lambda)-\log m_s(\lambda)$ gives
	$\partial_\lambda\log \PP_{s,\lambda}(N=n)=-(S-\E_{s,\lambda}[S])$; hence
	\[
	\partial_\lambda T_k
	=
	\partial_\lambda \E_{s,\lambda}[X_k]
	=
	\E_{s,\lambda}\big[X_k\,\partial_\lambda\log \PP_{s,\lambda}(N)\big]
	=
	-\Cov_{s,\lambda}(X_k,S)
	=
	-\,T_k\,M_k,
	\]
	where $M_k=\E_{s,\lambda}[S\mid N>k]-\E_{s,\lambda}[S]$ as in \eqref{eq:Mk-def-height}.
	Therefore
	\[
	\partial_\lambda w_k
	=
	(\partial_\lambda a_k)\,T_k+a_k\,\partial_\lambda T_k
	=
	a_kT_k(d_k-M_k)
	=
	-\,w_k\,\widetilde M_k,
	\qquad \widetilde M_k=M_k-d_k,
	\]
	and inserting this into \eqref{eq:bbar-ratio-derivative} yields \eqref{eq:bbar-derivative-height}.
	Finally, if $k\mapsto b_k^{(s)}$ and $k\mapsto \widetilde M_k^{(s)}(\lambda)$ are nondecreasing, then applying
	Lemma~\ref{lem:weighted-chebyshev-eventual} to the weights $w_k^{(s)}(\lambda)$ and the centered sequence
	$k\mapsto b_k^{(s)}-\bar b_s(\lambda)$ gives $\partial_\lambda \bar b_s(\lambda)\le 0$.

\end{proof}

\begin{lemma}\label{lem:cumulative-chebyshev-height}
Let $(w_k)_{k\ge0}$ be nonnegative weights with $0<W:=\sum_{k\ge0}w_k<\infty$, and let $(m_k)_{k\ge0}$ be a real sequence such that the sums below are finite. Write
\[
	\bar m:=\frac{1}{W}\sum_{k\ge0}w_km_k.
\]
Then
\[
	\sum_{k\ge0}(b_k-\bar b)w_km_k\ge0,
	\qquad
	\bar b:=\frac{1}{W}\sum_{k\ge0}w_kb_k,
\]
for every nondecreasing sequence $(b_k)_{k\ge0}$ with finite weighted first moment if and only if
\begin{equation}\label{eq:cumulative-chebyshev-tail}
	\sum_{k\ge\ell}w_k(m_k-\bar m)\ge0
	\qquad\text{for every }\ell\ge1.
\end{equation}
\end{lemma}

\begin{proof}
Apply Lemma~\ref{lem:increasing-dual-cone} to the signed measure
\[
\mu_k:=w_k(m_k-\bar m).
\]
It has total mass zero, and its upper-tail inequalities are exactly \eqref{eq:cumulative-chebyshev-tail}.  Pairing it with a nondecreasing $b$ gives
\[
\sum_{k\ge0}b_k\mu_k
=
\sum_{k\ge0}(b_k-\bar b)w_km_k,
\]
because constants pair to zero with $\mu$.
\end{proof}

The derivative formula gives useful sufficient criteria for Assumption~H. First, if for every
$s\in[0,1]$ and every $\lambda\in[\lambda_s,\lambda_s^\ast]$,
\begin{equation}\label{eq:height-cov-sufficient}
  \sum_{k\ge0}
  \big(\log h(k)-\bar L_s(\lambda)\big)
  w_k^{(s)}(\lambda)\widetilde M_k^{(s)}(\lambda)
  \ge0,
\end{equation}
then Lemma~\ref{lem:bbar-derivative-height} gives
$\partial_\lambda\bar L_s(\lambda)\le0$ on $[\lambda_s,\lambda_s^\ast]$, and hence the endpoint order in Assumption~H follows.

A stronger tail-form sufficient condition is the following. Define
\[
\langle \widetilde M\rangle_{s,\lambda}
:=
\frac{1}{W_s(\lambda)}
\sum_{j\ge0}w_j^{(s)}(\lambda)\widetilde M_j^{(s)}(\lambda).
\]
If, for every $s\in[0,1]$, every $\lambda\in[\lambda_s,\lambda_s^\ast]$, and every $\ell\ge1$,
\begin{equation}\label{eq:height-tail-sufficient}
\sum_{k\ge\ell}w_k^{(s)}(\lambda)
\Big(\widetilde M_k^{(s)}(\lambda)-\langle \widetilde M\rangle_{s,\lambda}\Big)
\ge 0,
\end{equation}
then Lemma~\ref{lem:cumulative-chebyshev-height}, applied with
$m_k=\widetilde M_k^{(s)}(\lambda)$ and $b_k=\log h(k)$, implies
\eqref{eq:height-cov-sufficient}.
The identity
$\partial_\lambda w_k^{(s)}(\lambda)=-w_k^{(s)}(\lambda)\widetilde M_k^{(s)}(\lambda)$ also gives
\[
	\partial_\lambda
	\frac{\sum_{k\ge\ell}w_k^{(s)}(\lambda)}{W_s(\lambda)}
	=
	-\frac{1}{W_s(\lambda)}
	\sum_{k\ge\ell}w_k^{(s)}(\lambda)
		\Big(\widetilde M_k^{(s)}(\lambda)-\langle\widetilde M\rangle_{s,\lambda}\Big).
\]
Thus these tail inequalities are precisely stochastic monotonicity of the height weights on the whole interval, and they are a useful sufficient condition for Assumption~H.

\subsection{A one-step criterion for monotonicity of \texorpdfstring{$\widetilde M_k$}{M-tilde k}}

Fix $s\in[0,1]$ and $\lambda>\lambda_s$. In this subsection we work under the law $\PP_{s,\lambda}$ from \eqref{eq:rootlaw-height} and suppress the
dependence on $(s,\lambda)$ in the notation. Recall
\[
c_k=\frac{1}{f_s(k)+\lambda},\qquad
S=\sum_{j=0}^{N-1}c_j,\qquad
T_k=\PP(N>k),\qquad
w_k=\lambda\,c_k\,T_k.
\]
Set
\begin{equation*}
\mu_k \;:=\; \E\big[S \,\big|\, N>k\big],
\qquad\text{so that}\qquad
M_k=\mu_k-\E[S].
\end{equation*}
Moreover, recall that the sign of $R_s'$ is governed by the monotonicity of
\[
\bar M_k:=M_k+c_k,
\]
equivalently of $\widetilde M_k$.
\begin{lemma}\label{lem:onestep-criterion}
Define the exit hazard
\begin{equation*}
q_k \;:=\; \PP(N=k+1 \mid N>k)\;=\;1-\frac{T_{k+1}}{T_k},
\end{equation*}
and the conditional residual tail sum
\begin{equation*}
\mathcal R_{k+1}\;:=\;\E\left[\sum_{j=k+1}^{N-1}c_j \ \bigg|\ N>k+1\right].
\end{equation*}
Then
\begin{equation}\label{eq:mu-increment}
\mu_{k+1}-\mu_k \;=\; q_k\,\mathcal R_{k+1}.
\end{equation}
Consequently,
\begin{equation}\label{eq:barM-increment}
\bar M_{k+1}-\bar M_k
\;=\;
q_k\,\mathcal R_{k+1}\;-\;(c_k-c_{k+1}),
\end{equation}
and in particular $\big(\bar M_k\big)_{k\ge0}$ is nondecreasing if and only if
\begin{equation}\label{eq:onestep-ineq}
q_k\,\mathcal R_{k+1}\ \ge\ c_k-c_{k+1}\qquad\text{for all }k\ge0.
\end{equation}
\end{lemma}

\begin{proof}
Condition on $N>k$. On the event $\{N=k+1\}$ we have $S=\sum_{j=0}^k c_j$; on the event $\{N>k+1\}$ we can
decompose
\[
S=\sum_{j=0}^k c_j+\sum_{j=k+1}^{N-1}c_j,
\qquad\text{and hence}\qquad
\E[S\mid N>k+1]=\sum_{j=0}^k c_j+\mathcal R_{k+1}.
\]
Therefore,
\[
\mu_k
=\E[S\mid N>k]
=q_k\sum_{j=0}^k c_j+(1-q_k)\Big(\sum_{j=0}^k c_j+\mathcal R_{k+1}\Big)
=\sum_{j=0}^k c_j+(1-q_k)\mathcal R_{k+1}.
\]
On the other hand, conditioning on $N>k+1$ gives
\[
\mu_{k+1}=\E[S\mid N>k+1]=\sum_{j=0}^k c_j+\mathcal R_{k+1}.
\]
Subtracting yields \eqref{eq:mu-increment}. Since $M_{k+1}-M_k=\mu_{k+1}-\mu_k$, we obtain
\eqref{eq:barM-increment} by adding $c_{k+1}-c_k$, and \eqref{eq:onestep-ineq} follows.
\end{proof}

\begin{remark}Lemma~\ref{lem:onestep-criterion} gives a convenient sufficient route to Assumption~H. If
\eqref{eq:onestep-ineq} holds for all $k$, then $k\mapsto\widetilde M_k^{(s)}(\lambda)$ is nondecreasing, and the ordinary weighted Chebyshev inequality implies the sufficient tail inequalities \eqref{eq:height-tail-sufficient}. In the sublinear regularly varying regime, Lemma~\ref{lem:rv-height} verifies
\eqref{eq:onestep-ineq} for all sufficiently large $k$ (uniformly over compact $\lambda$--windows). This is the role of regular variation in the optional verification criterion; it is not a standing assumption in Theorem~\ref{thm:height-monotonicity}.
The finitely many remaining one-step inequalities are therefore a checkable sufficient condition for Assumption~H, but the theorem itself uses only the weaker height-score endpoint order.
\end{remark}

\begin{lemma}\label{lem:rv-height}
	Assume $f_s$ is eventually nondecreasing and regularly varying with index $\rho\in[0,1)$.
	Fix $s$ and set the compact interval $I_s:=[\lambda_s,\lambda^{\ast}_s]$.
	Then there exists $k_0=k_0(s)$ such that for every $\lambda\in I_s$ the one-step inequality
	\eqref{eq:onestep-ineq} holds for all $k\ge k_0$.
	Consequently, for each $\lambda\in I_s$ the corrected profile $k\mapsto \widetilde M_k^{(s)}(\lambda)$
	is eventually nondecreasing.
\end{lemma}

\begin{proof}
	Fix $s$ and write $f=f_s$. Set $\lambda_{\mathrm{lo}}:=\lambda_s$ and $\lambda_{\mathrm{hi}}:=\lambda^{\ast}_s$.
	Throughout we consider $\lambda\in[\lambda_{\mathrm{lo}},\lambda_{\mathrm{hi}}]$ and constants may depend on $(s,\lambda_{\mathrm{lo}},\lambda_{\mathrm{hi}})$ but not on $k$.

	Since $f$ is (eventually) nondecreasing and regularly varying with index $\rho\in[0,1)$, we may write
	$f(x)=x^\rho L(x)$ with $L$ slowly varying. We use the following standard facts\cite{Bingham1987}:
	\begin{enumerate}[label=(RV\arabic*)]
		\item \label{rv:ratio1} $f(k+1)/f(k)\to 1$ and hence $f(k+1)-f(k)=o(f(k))$.
		\item \label{rv:uct} (Uniform convergence on $o(k)$-windows) If $m_k=o(k)$, then
		$\sup_{0\le j\le m_k} f(k+j)/f(k)\to 1$.
		\item \label{rv:potter} (Potter bound) For each $\eta>0$ there exist $k_1$ and $C_\eta$ such that for all $k\ge k_1$
		and all $\ell\ge0$,
		\[
		\frac{f(k+\ell)}{f(k)}\le C_\eta\Big(1+\frac{\ell}{k}\Big)^{\rho+\eta}.
		\]
	\end{enumerate}

	Recall the notation from Lemma~\ref{lem:onestep-criterion} (here the dependence on $(s,\lambda)$ is suppressed):
	\[
	c_k=\frac{1}{f(k)+\lambda},\qquad
	q_k=\PP(N=k+1\mid N>k),\qquad
	\mathcal R_{k+1}=\E\left[\sum_{j=k+1}^{N-1}c_j\,\Big|\,N>k+1\right].
	\]
	By Lemma~\ref{lem:onestep-criterion}, it suffices to prove that for all large $k$ (uniformly in $\lambda\in[\lambda_{\mathrm{lo}},\lambda_{\mathrm{hi}}]$),
	\begin{equation}\label{eq:rv-target-unif}
		q_k\,\mathcal R_{k+1}\ \ge\ c_k-c_{k+1}.
	\end{equation}

	\medskip
	\noindent\textit{Step 1: $c_k-c_{k+1}$ is uniformly negligible compared to $c_{k+1}$.}
	By \ref{rv:ratio1},
	\[
	\frac{c_k-c_{k+1}}{c_{k+1}}
	=
	\frac{f(k+1)-f(k)}{f(k)+\lambda}
	\le
	\frac{f(k+1)-f(k)}{f(k)+\lambda_{\mathrm{lo}}}
	\longrightarrow 0.
	\]
	Hence for any $\eta\in(0,1)$ there exists $k_1(\eta)$ such that for all $k\ge k_1(\eta)$ and all $\lambda\in[\lambda_{\mathrm{lo}},\lambda_{\mathrm{hi}}]$,
	\begin{equation}\label{eq:ck-diff-small-unif}
		c_k-c_{k+1}\ \le\ \eta\,c_{k+1}.
	\end{equation}

	\medskip
	\noindent\textit{Step 2: uniform lower bound on $q_k$.}
	Write $A_n(\lambda)=\prod_{i=0}^{n-1}\frac{f(i)}{f(i)+\lambda}$ and $r_k(\lambda)=\sum_{n\ge k+1}A_n(\lambda)$.
	Then
	\[
	q_k=\frac{A_{k+1}}{r_k}=\frac{1}{1+H_k},\qquad
	H_k:=\sum_{\ell\ge1}\frac{A_{k+1+\ell}}{A_{k+1}}=\sum_{\ell\ge1}\prod_{j=1}^{\ell}\frac{f(k+j)}{f(k+j)+\lambda}.
	\]
	Using $\log(1-x)\le -x$ and monotonicity of $f$ gives
	\[
	\prod_{j=1}^{\ell}\frac{f(k+j)}{f(k+j)+\lambda}
	\le
	\exp\Big(-\frac{\lambda\,\ell}{f(k+\ell)+\lambda}\Big).
	\]
	Since $\lambda\mapsto \lambda/(f(k+\ell)+\lambda)$ is increasing, the right-hand side is maximised at $\lambda=\lambda_{\mathrm{lo}}$,
	hence $H_k(\lambda)\le H_k(\lambda_{\mathrm{lo}})$ and it suffices to bound $H_k(\lambda_{\mathrm{lo}})$.
	We spell out the Potter summation. Choose $\eta>0$ so that
	$\beta:=\rho+\eta<1$. For all large $k$,
	\[
	f(k+\ell)+\lambda_{\mathrm{lo}}
	\le C_0\,(f(k)+\lambda_{\mathrm{lo}})\Big(1+\frac{\ell}{k}\Big)^\beta,
	\qquad \ell\ge0.
	\]
	Splitting the sum defining $H_k$ at $\ell=k$ gives
	\[
	\sum_{\ell=1}^{k}
	\exp\Big(-\frac{\lambda_{\mathrm{lo}}\,\ell}{f(k+\ell)+\lambda_{\mathrm{lo}}}\Big)
	\le
	\sum_{\ell=1}^{\infty}
	\exp\Big(-\frac{c\,\ell}{f(k)+\lambda_{\mathrm{lo}}}\Big)
	\le C_1\Big(1+\frac{f(k)}{\lambda_{\mathrm{lo}}}\Big).
	\]
	For $\ell>k$, regular variation with index $<\beta$ gives
	$f(k)+\lambda_{\mathrm{lo}}\le C_2 k^\beta$ for all large $k$, hence
	\[
	f(k+\ell)+\lambda_{\mathrm{lo}}
	\le C_3 \ell^\beta,
	\]
	and consequently
	\[
	\sum_{\ell>k}
	\exp\Big(-\frac{\lambda_{\mathrm{lo}}\,\ell}{f(k+\ell)+\lambda_{\mathrm{lo}}}\Big)
	\le
	\sum_{\ell>k}\exp(-c'\ell^{1-\beta})
	\le C_4.
	\]
	Thus
	\[
	H_k(\lambda_{\mathrm{lo}})\le C_5\Big(1+\frac{f(k)}{\lambda_{\mathrm{lo}}}\Big)\qquad\text{for all large $k$},
	\]
	and therefore for all large $k$ and all $\lambda\in[\lambda_{\mathrm{lo}},\lambda_{\mathrm{hi}}]$,
	\begin{equation}\label{eq:qk-lower-unif}
		q_k=\frac{1}{1+H_k}\ \ge\ \frac{c_0\,\lambda_{\mathrm{lo}}}{f(k)+\lambda_{\mathrm{hi}}}.
	\end{equation}

	\medskip
	\noindent\textit{Step 3: uniform lower bound on $\mathcal R_{k+1}$.}
	Let $L:=N-(k+1)$ under the conditional law $N>k+1$, so
	\[
	\mathcal R_{k+1}
	=
	\E\Big[\sum_{j=0}^{L-1}c_{k+1+j}\ \Big|\ N>k+1\Big].
	\]
	Choose
	\[
	M_k:=\Big\lfloor \delta\,\frac{f(k)}{\lambda_{\mathrm{hi}}}\Big\rfloor
	\]
	with fixed $\delta\in(0,1)$. If $f$ is bounded, then the eventual monotonicity of $f$ gives a finite limit, hence $f(k+1)-f(k)\to0$ and $c_k-c_{k+1}\to0$ uniformly in $\lambda\in[\lambda_{\mathrm{lo}},\lambda_{\mathrm{hi}}]$.
	Moreover Step~2 gives $q_k\ge q_0>0$ and, under $N>k+1$, the residual sum always contains the term $c_{k+1}$, so $\mathcal R_{k+1}\ge c_{k+1}\ge c_0'>0$.
	Thus \eqref{eq:rv-target-unif} follows for all large $k$ in the bounded case. We may therefore assume $f(k)\to\infty$, in which case $M_k\to\infty$ and $M_k=o(k)$ by $\rho<1$.
	By \ref{rv:uct}, for all large $k$,
	\[
	\sup_{0\le j\le M_k}\frac{f(k+1+j)}{f(k+1)}\le 2.
	\]
	Hence
	\[
	c_{k+1+j}=\frac{1}{f(k+1+j)+\lambda}\ \ge\ \frac{1}{2}\,c_{k+1}
	\quad(0\le j\le M_k),
	\]
	uniformly in $\lambda\in[\lambda_{\mathrm{lo}},\lambda_{\mathrm{hi}}]$. Therefore,
	\begin{equation}\label{eq:R-lower-trunc-unif}
		\mathcal R_{k+1}\ \ge\ \frac{1}{2}\,c_{k+1}\,\E[L\wedge M_k\mid N>k+1].
	\end{equation}

	Next, for $1\le \ell\le M_k$, monotonicity of $f$ and $\log(1+u)\le u$ imply
	\[
	\frac{A_{k+1+\ell}}{A_{k+2}}
	=
	\prod_{j=1}^{\ell-1}\frac{1}{1+\lambda/f(k+1+j)}
	\ge
	\exp\Big(-\lambda\sum_{j=1}^{\ell-1}\frac{1}{f(k+1+j)}\Big)
	\ge
	\exp\Big(-\lambda_{\mathrm{hi}}\,\frac{\ell}{f(k)}\Big)
	\ge e^{-\delta},
	\]
	by the choice of $M_k$. Hence $A_{k+1+\ell}\ge e^{-\delta}A_{k+2}$ for all $1\le \ell\le M_k$, uniformly in $\lambda\in[\lambda_{\mathrm{lo}},\lambda_{\mathrm{hi}}]$.
	Using $r_{k+1}=A_{k+2}(1+H_{k+1})$ and the bound $H_{k+1}\le C_0(1+f(k)/\lambda_{\mathrm{lo}})$ from Step~2,
	we obtain for $1\le \ell\le M_k$ and all large $k$,
	\[
	\PP(L=\ell\mid N>k+1)=\frac{A_{k+1+\ell}}{r_{k+1}}
	\ge
	\frac{e^{-\delta}}{C_1(1+f(k)/\lambda_{\mathrm{lo}})}
	\ge
	c_1\,\frac{\lambda_{\mathrm{lo}}}{f(k)+\lambda_{\mathrm{hi}}}.
	\]
	Therefore,
	\[
	\E[L\wedge M_k\mid N>k+1]
	\ge
	\sum_{\ell=1}^{M_k}\ell\,\PP(L=\ell\mid N>k+1)
	\ge
	c_2\,\frac{\lambda_{\mathrm{lo}}}{f(k)+\lambda_{\mathrm{hi}}}\,M_k^2
	\asymp
	c_3\,\frac{f(k)}{\lambda_{\mathrm{hi}}},
	\]
	for all large $k$, uniformly in $\lambda\in[\lambda_{\mathrm{lo}},\lambda_{\mathrm{hi}}]$.
	Combining with \eqref{eq:R-lower-trunc-unif} yields
	\begin{equation}\label{eq:R-lower-final-unif}
		\mathcal R_{k+1}\ \ge\ c_4\,c_{k+1}\,\frac{f(k)}{\lambda_{\mathrm{hi}}}
		\qquad\text{for all large $k$, uniformly in }\lambda\in[\lambda_{\mathrm{lo}},\lambda_{\mathrm{hi}}].
	\end{equation}

	\medskip
	\noindent\textit{Step 4: conclude the one-step inequality.}
	Multiplying \eqref{eq:qk-lower-unif} and \eqref{eq:R-lower-final-unif} gives, for all large $k$ and uniformly in $\lambda\in[\lambda_{\mathrm{lo}},\lambda_{\mathrm{hi}}]$,
	\[
	q_k\,\mathcal R_{k+1}\ \ge\ c_5\,c_{k+1}.
	\]
	Choose $\eta:=c_5/2$ in \eqref{eq:ck-diff-small-unif}; then for all sufficiently large $k$ (uniformly in $\lambda\in[\lambda_{\mathrm{lo}},\lambda_{\mathrm{hi}}]$),
	\[
	q_k\,\mathcal R_{k+1}\ \ge\ c_5\,c_{k+1}\ \ge\ c_k-c_{k+1},
	\]
	which is \eqref{eq:rv-target-unif}. The final claim follows from Lemma~\ref{lem:onestep-criterion}.
\end{proof}

\subsection[Conclusion of height monotonicity]{Conclusion of Theorem~\ref{thm:height-monotonicity}}

Fix a point $s\in(0,1)$ at which the derivatives in Lemma~\ref{lem:gamma-derivative-height} exist and the envelope identity in (II.3) holds. By Lemma~\ref{lem:gamma-derivative-height},
\begin{equation}\label{eq:gamma-derivative-conclusion}
	R_s'
	=
	R_s\big(\bar b_s(\lambda_s)-\bar b_s(\lambda^{\ast}_s)\big),
\end{equation}
where $\bar b_s(\lambda)$ is the weighted mean \eqref{eq:bbar-def-height}.
Since $b_k^{(s)}=\log h(k)$, this weighted mean is exactly $\bar L_s(\lambda)$ from Assumption~H.
Therefore Assumption~H gives
$\bar b_s(\lambda_s)\ge\bar b_s(\lambda^{\ast}_s)$.
In the stronger stochastic-order formulation, this is the same dual-cone pairing from Lemma~\ref{lem:increasing-dual-cone} applied to
$\nu_{s,\lambda_s}-\nu_{s,\lambda_s^\ast}$ and the increasing tangent $\log h$.
Returning to \eqref{eq:gamma-derivative-conclusion} gives $R_s'\ge0$ at every such $s$.

The Euler--Lotka equation, (I.4), and the domination in (II.2) imply by the implicit function theorem that $s\mapsto\lambda_s$ is continuously differentiable, while (II.3) assumes that $s\mapsto\kappa_s$ is absolutely continuous. Hence $R_s=\lambda_s\kappa_s$ is absolutely continuous, and the preceding inequality holds for almost every $s$. Integrating over $s$ yields $R_1\ge R_0$, i.e.
\[
\lambda_f\kappa_f=R_f\ \ge\ R_g=\lambda_g\kappa_g,
\]
and hence
\[
\mathsf{c}^{\ast}_f=\frac{1}{R_f}\ \le\ \frac{1}{R_g}=\mathsf{c}^{\ast}_g.
\]
This completes the proof of Theorem~\ref{thm:height-monotonicity}.

\appendix

\section{Proof of the continuous-time embedding}
\label{app:continuous-time-embedding}

\begin{proof}[Proof of Lemma~\ref{lem:ct-embedding}]
Conditionally on the current finite tree $T$, the next birth time is the minimum of independent exponential clocks
with rates $f(\deg_T(x))$, one clock at each vertex $x\in T$. Hence the minimum clock is attached to $x$ with
probability
\[
\frac{f(\deg_T(x))}{\sum_{y\in T}f(\deg_T(y))},
\]
and after that birth the new child is added at $x$. This is exactly the transition rule of the discrete PA tree, and induction over the jump chain gives the claim.
\end{proof}

\section{On the analytical regularity conditions}
\label{app:analytic-regularity}

\begin{proof}[Proof of \eqref{eq:mfprime-tail}]
Fix $\varphi$ and write
\(
A_n(\lambda)=\prod_{i=0}^{n-1}\frac{\varphi(i)}{\varphi(i)+\lambda}.
\)
For each fixed $n\ge1$, differentiation of the finite product gives
\[
\partial_\lambda A_n(\lambda)
=
-A_n(\lambda)\sum_{i=0}^{n-1}\frac{1}{\varphi(i)+\lambda}.
\]
Consequently, for any $\lambda>0$,
\[
-m_\varphi'(\lambda)=\sum_{n\ge1}A_n(\lambda)\sum_{i=0}^{n-1}\frac{1}{\varphi(i)+\lambda},
\]
with the understanding that the series may take the value $+\infty$. Since the summand is nonnegative, Tonelli's theorem yields
\[
-m_\varphi'(\lambda)=\sum_{i\ge0}\frac{1}{\varphi(i)+\lambda}\sum_{n\ge i+1}A_n(\lambda)
=\sum_{i\ge0}\frac{r_i(\lambda)}{\varphi(i)+\lambda},
\]
again as an identity in $[0,\infty]$, where
\(
r_i(\lambda):=\sum_{n\ge i+1}A_n(\lambda).
\)
In particular, at $\lambda=\lambda_\varphi$ the finiteness assumption
$-m_\varphi'(\lambda_\varphi)\in(0,\infty)$ implies that the right-hand side is finite.
\end{proof}

\begin{proof}[Dominated convergence under \eqref{eq:uniform-dominating}]
Let $\lambda_0<\lambda_-$ be as in (II.2), put $\bar\lambda=(\lambda_0+\lambda_-)/2$, and choose $\varepsilon>0$ so that $\delta:=\lambda_-(1-\varepsilon)-\bar\lambda>0$. For $\mu\ge\lambda_-(1-\varepsilon)$ abbreviate
$S_n(\mu):=\sum_{i<n}(f_\theta(i)+\mu)^{-1}$ and
$L_n(\mu):=\sum_{i<n}(1+|\log h(i)|)(f_\theta(i)+\mu)^{-1}$.
Then $1-x\le e^{-x}$ gives
\[
\frac{A_n^{(\theta)}(\mu)}{A_n^{(\theta)}(\bar\lambda)}
=\prod_{i<n}\left(1-\frac{\mu-\bar\lambda}{f_\theta(i)+\mu}\right)
\le e^{-\delta S_n(\mu)}.
\]
Since $\sup_{x\ge0}xe^{-\delta x}<\infty$, monotonicity in the Laplace parameter yields
\begin{align*}
A_n^{(\theta)}(\mu)L_n(\mu)(1+S_n(\mu))
&\le C_\delta A_n^{(\theta)}(\lambda_0)L_n(\lambda_0),\\
\sum_{n\ge1}A_n^{(\theta)}(\lambda_0)L_n(\lambda_0)
&=\sum_{i\ge0}\frac{(1+|\log h(i)|)r_i^{(\theta)}(\lambda_0)}{f_\theta(i)+\lambda_0}<\infty,
\end{align*}
uniformly in $\theta$, by \eqref{eq:uniform-dominating}.

Omitting $1+S_n$ gives the first-derivative bounds needed for the implicit function theorem. After gauging, $A_n^{(\theta,*)}(\lambda)=A_n^{(\theta)}(\lambda_\theta\lambda)$, while $\lambda_\theta$ and $a'(\theta)=-\lambda_\theta'/\lambda_\theta$ are locally bounded by Lemma~\ref{lem:lambda-regularity}; the displayed bound therefore also controls the mixed derivative. Dominated convergence now justifies the continuity, mixed differentiation, and absolutely convergent rearrangements in Lemma~\ref{lem:Qprime}.
\end{proof}

\begin{lemma}
Let $\varphi:\N_0\to(0,\infty)$ and $\lambda>0$. Define
\[
  A_n(\lambda):=\prod_{i=0}^{n-1}\frac{\varphi(i)}{\varphi(i)+\lambda},\qquad A_0(\lambda):=1.
\]
The choice $A_0(\lambda)=1$ corresponds to the empty product, so the sum in \eqref{eq:telescoping} naturally starts at $n=0$. Then for every $m\ge0$,
\begin{equation}\label{eq:telescoping}
\sum_{n=0}^m \frac{\lambda}{\varphi(n)+\lambda}\,A_n(\lambda)=1-A_{m+1}(\lambda).
\end{equation}
If moreover $A_{m}(\lambda)\to 0$ as $m\to\infty$, then
\[
\sum_{n=0}^\infty \frac{\lambda}{\varphi(n)+\lambda}\,A_n(\lambda)=1.
\]
\end{lemma}

\begin{proof}
Since $A_{n+1}(\lambda)=A_n(\lambda)\frac{\varphi(n)}{\varphi(n)+\lambda}$,
\[
A_n(\lambda)-A_{n+1}(\lambda)=A_n(\lambda)\frac{\lambda}{\varphi(n)+\lambda}.
\]
Summing over $n=0,\dots,m$ telescopes to \eqref{eq:telescoping}.
\end{proof}


\noindent\textbf{A note on the use of AI.}

\noindent The accompanying Python file \texttt{counterexample\_certificate.py} was created with the assistance of OpenAI Codex (GPT-5.6) to implement the exact-arithmetic certificate for the counterexample.

The author reviewed and verified the code and its output and takes full responsibility for the computational claims.

\medskip
\noindent\textbf{Acknowledgement.} I would like to thank Bas Lodewijks and the organisers of the Sheffield Random Networks Workshop 2025 for raising my awareness of the depth monotonicity problem in random recursive trees.

\printbibliography

\end{document}